\documentclass[12pt]{amsart}
\usepackage{times}
\usepackage[T1]{fontenc}
\usepackage{dsfont}
\usepackage{mathrsfs}
\usepackage[colorlinks]{hyperref}
\usepackage{xcolor}
\usepackage[a4paper,asymmetric]{geometry}
\usepackage{mathscinet}
\usepackage{fullpage}
\usepackage{latexsym}
\usepackage{amsthm}
\usepackage{amssymb}
\usepackage{amsfonts}
\usepackage{amsmath}
\usepackage{graphicx}
\newtheorem{theorem}{Theorem}[section]
\newtheorem{thm}[theorem]{Theorem}

\newtheorem{lem}[theorem]{Lemma}
\newtheorem{proposition}[theorem]{Proposition}

\newtheorem{corollary}[theorem]{Corollary}

\theoremstyle{definition}

\newtheorem{defn}[theorem]{Definition}

\theoremstyle{remark}

\newtheorem{rem}[theorem]{Remark}
\numberwithin{equation}{section}

 \DeclareMathAlphabet{\mathpzc}{OT1}{pzc}{m}{it}

  \newcommand{\dif}{\mathrm{d}}

 \newcommand{\E}{\mathbb{E}}            
 
 \newcommand{\e}{\varepsilon}
 
 \newcommand{\p}{\partial}

 \newcommand{\N}{\mathbb{N}}
 \newcommand{\R}{\mathbb{R}}
 
 \newcommand{\PP}{\mathbb{P}}
 \newcommand{\mcl}{\mathcal}
 
 \newcommand{\Be}{\begin{equation}}
 \newcommand{\Ee}{\end{equation}}
 \newcommand{\Bs}{\begin{split}}
 \newcommand{\Es}{\end{split}}
  \newcommand{\Bes}{\begin{equation*}}
 \newcommand{\Ees}{\end{equation*}}
 \newcommand{\BT}{\begin{thm}}
 \newcommand{\ET}{\end{thm}}
 \newcommand{\Bp}{\begin{proof}}
 \newcommand{\Ep}{\end{proof}}
 \newcommand{\BL}{\begin{lem}}
 \newcommand{\EL}{\end{lem}}
 \newcommand{\BP}{\begin{proposition}}
 \newcommand{\EP}{\end{proposition}}
 \newcommand{\BC}{\begin{corollary}}
 \newcommand{\EC}{\end{corollary}}
 \newcommand{\BR}{\begin{rem}}
 \newcommand{\ER}{\end{rem}}
 \newcommand{\BD}{\begin{defn}}
 \newcommand{\ED}{\end{defn}}
 \newcommand{\BI}{\begin{itemize}}
 \newcommand{\EI}{\end{itemize}}
 
 \newcommand{\tl}{\tilde}

\begin{document}
\title[Approximation of stable law by Stein's method]
{Approximation of stable law in Wasserstein-1 distance by Stein's method}

\author[L. Xu]{Lihu Xu}
\address{1. Department of Mathematics,
Faculty of Science and Technology
University of Macau
Av. Padre Tom\'{a}s Pereira, Taipa
Macau, China; \ \ 2. UM Zhuhai Research Institute, Zhuhai, China.}
\email{lihuxu@umac.mo}


\begin{abstract} \label{abstract}
Let $n \in \N$, let $\zeta_{n,1},...,\zeta_{n,n}$ be a sequence of independent random variables with $\E \zeta_{n,i}=0$ and $\E |\zeta_{n,i}|<\infty$ for each $i$, and let $\mu$ be an $\alpha$-stable distribution having characteristic function $e^{-|\lambda|^{\alpha}}$ with $\alpha\in (1,2)$. Denote $S_{n}=\zeta_{n,1}+...+\zeta_{n,n}$ and its distribution by $\mcl L(S_n)$, we bound the Wasserstein-1 distance of $\mcl L(S_{n})$ and $\mu$ essentially by an $L^{1}$ discrepancy between two kernels. More precisely, we prove the following inequality:
\ \ \
\Bes
\begin{split}
d_W\left(\mcl L (S_n), \mu\right) \ \le C  \left[\sum_{i=1}^n\int_{-N}^N \left|\frac{\mcl K_\alpha(t,N)}n -\frac{ K_i(t,N)}{\alpha}\right| \dif t \ +\ \mcl R_{N,n}\right],
\end{split}
\Ees
where $d_W$ is the Wasserstein-1 distance of probability measures, $\mcl K_\alpha(t,N)$ is the kernel of a decomposition of the fractional Laplacian $\Delta^{\frac \alpha2}$, $ K_i(t,N)$ is a $K$ function \cite{ChGoSh11} with a truncation, and $\mcl R_{N,n}$ is a small remainder. The integral term
\ \ \
$$\sum_{i=1}^n\int_{-N}^N \left|\frac{\mcl K_\alpha(t,N)}n -\frac{ K_i(t,N)}{\alpha}\right| \dif t$$
can be interpreted as an $L^{1}$ discrepancy.

As an application, we prove a general theorem of stable law convergence rate when $\zeta_{n,i}$ are i.i.d. and the distribution falls in the normal domain of attraction
of $\mu$. To test our results, we compare our convergence rates with those known in the literature for four given examples, among which the distribution in the fourth example is not in the normal domain of attraction of $\mu$.
\end{abstract}

\maketitle
\section{Introduction}
Let $n \in \N$ and let $\zeta_{n,1},...,\zeta_{n,n}$ be a sequence of independent random variables with $\E \zeta_{n,i}=0$ for each $i$,
denote
$$S_n\ =\ \zeta_{n,1}+...+\zeta_{n,n}.$$
It is well known that $S_n$ weakly converges to the standard normal distribution $\Phi$ if this sequence satisfies
the Lindeberg condition and $\E S^2_n \rightarrow 1$. If we further assume that $\E|\zeta_{n,i}|^3<\infty$ for each $i$, then Berry-Esseen theorem follows
\Bes
\sup_{x \in \R} \left|\PP\left(S_n \le x\right)-\Phi(x)\right| \ \le \ C \sum_{i=1}^n \E|\zeta_{n,i}|^3,
\Ees
where $C>0$ is some constant.

Stein's method was put forward in the seminal work \cite{Stein72} to study normal approximations such as Berry-Esseen theorem, very soon thereafter Chen applied this method to get the convergence rate of the Poisson approximation \cite{Chen75}. Nowadays, Stein's method has been extended and refined by many authors and become a very important tool for getting bounds of measure approximations, see \cite{BaCeXi07, BrDa17,ChaSha11, Do14, EiLo10, GoTi06, KuTu11, PeRoRo12, Fang14,GaPiRe17,GoRe97,Hsu05,NoPe09,NoPe12, NoPeSw14,ReRo09}. For more references, we refer the reader to the webpages: https://sites.google.com/site/steinsmethod/home and https://sites.google.com/site/malliavinstein/home.

The stable distribution is one of the most important distributions in probability theory and has a lot of applications in economics, finance, physics and so on, see the monographs \cite{HauLus15, UcZo99} and the references therein for details. If the above sequence $\{\zeta_{n,i}\}_{1 \le i \le n}$ are assumed to have a suitable heavy tail, $S_n$ weakly converges to a stable distribution \cite[Theorem 3.7.2]{Dur10}. However, it seems that there are not many
results about the rate of stable law convergence, see \cite{BaSh60, BoSh70, BuHa78, ChWo92, DaNa02, Hall81,KuKe00,JuPa98}. Moreover, all these works are proved by the characteristics function method in Kolmogorov distance.
\vskip 3mm

 The goal of this paper is to study the $\alpha$-stable law approximation in Wassertein-1 distance (it is often called $W_1$ distance or $L^1$ distance for simplicity) by Stein's method for $\alpha \in (1,2)$.  We prove two general theorems, one is a framework which gives a general bound for the $W_1$ distance between $S_{n}$ and $\mu$, the other is an application of the framework when $\{\zeta_{n,i}\}_{1 \le i \le n}$ are i.i.d and their distribution falls in the normal domain of attraction of $\mu$. It  should be stressed that some known results can give the rate for $\alpha \in (0,1]$, while ours is only for $\alpha \in (1,2)$. The reason is stable distributions do not have 1st moment for $\alpha \in (0,1]$, and the $W_{1}$ distance is consequently NOT well defined in general. Therefore, our assumption $\alpha \in (1,2)$ is \emph{essential}.

 We apply the two theorems to four examples which have been studied by several authors \cite{KuKe00,JuPa98,DaNa02,ChZh16,Hall81} in Kolmogorov distance, and compare our convergence rates with theirs. A big advantage of our theorems is that one can obtain an explicit bound of convergence rather than only giving the order of rates as in the known literatures.
\vskip 3mm

Our first example is a sequence of i.i.d. random variables having a Pareto distribution density $p(x)=\frac{\alpha}{2|x|^{\alpha+1}} 1_{\{|x|>1\}}$, whose sum scaled by $n^{-1/\alpha}$ weakly converges to a symmetric stable distribution. A convergence rate $n^{-\frac{2-\alpha}{\alpha}}$ in Kolmogorov distance was obtained in \cite{KuKe00}, while \cite{DaNa02} proved a rate $n^{-\frac{2-\alpha}{d+\alpha}}$ for $d$ dimensional stable law in total variation distance and conjectured that a better rate should be $n^{-\frac{2-\alpha}{\alpha}}$ in the $L^{1}$ or total variation distance. Our result gives a positive answer to their conjecture for the $L^{1}$ distance case when $d=1$.

The second example is from \cite{KuKe00,AMPS16}. The distribution of i.i.d. random variables in \cite{KuKe00} is a perturbed Pareto distribution with a density $p(x)=\left(\frac{A}{|x|^{\alpha+1}}+\frac{B}{|x|^{\beta+1}}\right)1_{\{|x| > a\}}$ for some $A>0$, $B>0$, $a>0$ and $\beta>\alpha$. We consider a more general distribution such that the distribution function  $F$ satisfies
$1-F(x)=\frac{A}{|x|^{\alpha}}+\frac{B_{1}(x)}{|x|^{\beta}}, \ F(-x)=\frac{A}{|x|^{\alpha}}+\frac{B_{2}(x)}{|x|^{\beta}}$ for large $x>0$, where $\beta>\alpha$, $A>0$, and $B_{1}(x), B_{2}(x)$ are bounded continuous functions. It seems that the technique in \cite{KuKe00} is not able to handle this general distribution case. In this paper, we obtain a convergence rate $n^{-\frac{2-\alpha}{\alpha}}$ for $\beta \in (2,\infty)$, while \cite{KuKe00} gives the same rate for $\beta \in (2 \alpha,\infty)$. Note that the example in \cite[Appendix B]{AMPS16} is covered by this one by taking $\beta=\alpha+1$.

 The third example is a special case of \cite{Hall81} by Hall. When the limit distribution is symmetric stable, we can get a rate $n^{-\frac{2-\alpha}{\alpha}}$ in some situations, while Hall obtained a rate $n^{-\beta}$ for some $0<\beta<\frac{2-\alpha}{\alpha}$.

The fourth example is from \cite{JuPa98}, the i.i.d. random variables therein have a density $p(x)=\frac{C (\log|x|)^\beta}{|x|^{\alpha+1}} 1_{\{|x|>c\}}$ with $\beta \in \R$ and $c, C>0$, which is not in the normal domain of attraction of a stable law. A convergence rate $(\log n)^{-1}$ in Kolmogorov distance was proved by a very delicate analysis depending on the special form of the distribution. Using our first general theorem, we can obtain a rate $(\log n)^{-1+\frac 1\alpha}$ in $W_1$ distance, which is worse than  $(\log n)^{-1}$. However, our theorem can be used to study more examples which can not be handled by the characteristics function method in \cite{JuPa98} directly. We defer to give the details of this example in the appendix.

 \vskip 3mm

 Let us now roughly explain the strategy of our method. In normal approximations, the $K$ function approach \cite{ChGoSh11} is to write
 \Be  \label{e:ESnFSn}
 \begin{split}
 \E[S_n f(S_n)]& \ =\ \sum_{i=1}^n \int_{-\infty}^\infty  \E[K_i(t) f'(S_n(i)+t)] \dif t, 
 \end{split}
 \Ee
 where $S_n(i)=S_n-\zeta_{n,i}$ and $K_i(t)=\E\left[\zeta_{n,i} 1_{\{0 \le t \le \zeta_{n,i}\}}-\zeta_{n,i} 1_{\{\zeta_{n,i}  \le t \le 0\}}\right]$, and bound its difference with $\E[f'(S_{n})]$.
\vskip 1mm

To prove the convergence rate of stable law, we shall find a solution $f$ of the Stein equation, \eqref{e:StEq} below, and bound
$$\E\left[\Delta^{\frac \alpha 2} f(S_n)-\frac 1\alpha S_n f'(S_n)\right],$$
where $\Delta^{\frac \alpha 2}$ is the fractional Laplacian defined by \eqref{e:FraLap} below. Inspired by the above observation of $\E[S_n f(S_n)]$, we represent
 \Be   \label{e:ESnFSn-1}
 \begin{split}
 \E[S_n f'(S_n)]& \ =\ \sum_{i=1}^n \int_{-N}^N  \E[ K_i(t,N) f''(S_n(i)+t)] \dif t+\mcl R,
 \end{split}
 \Ee
 where $N>0$ is an arbitrary number to be chosen later, $\mcl R$ is a remainder and
\Be  \label{e:Ker1-1}
 K_i(t,N)\ =\ \E\left[\zeta_{n,i} 1_{\{0 \le t \le \zeta_{n,i}  \le N\}}-\zeta_{n,i} 1_{\{-N \le  \zeta_{n,i}  \le t \le 0\}}\right].
\Ee
Due to the heavy tail property of $\zeta_{n,i}$, we need to truncate $\zeta_{n,i}$ and thus get a remainder $\mcl R$. On the other hand, we decompose $\Delta^{\frac \alpha 2} f$ into a linear combination of $f''$ with a remainder $\mcl R'$ as the following:
\ \ \ \
\Be  \label{e:DecFx}
\Delta^{\frac \alpha 2} f(x)\ =\ \int_{-N}^N \mcl K_{\alpha}(t,N) f''(x+t) \dif t + \mcl R',
\Ee
where
\Be \label{e:StaKer}
\mcl K_\alpha(t,N)=\frac{d_\alpha}{\alpha(\alpha-1)}\left(|t|^{1-\alpha}-N^{1-\alpha}\right) \ \ \ \ {\rm with} \ \ \ \  d_\alpha=\left(\int_{-\infty}^\infty \frac{1-\cos y}{|y|^{1+\alpha}} \dif y\right)^{-1}.
\Ee
Using \eqref{e:DecFx} and \eqref{e:ESnFSn-1}, we see
\Be  \label{e:StaDec}
\begin{split}
\E\left[\Delta^{\frac \alpha 2} f(S_n)-\frac 1\alpha S_n f'(S_n)\right] & \ =\ \sum_{i=1}^n \!\! \int_{-N}^N \!\E\left[\left(\frac{\mcl K_\alpha(t,N)} n-\frac{ K_i(t,N)}{\alpha}\right) f''(S_n(i)+t)\right] \dif t+\mcl R'', \\
\end{split}
\Ee
where $\mcl R''$ is another remainder. Hence,
\ \ \ \
\begin{equation*}
\left|\E\left[\Delta^{\frac \alpha 2} f(S_n)-\frac 1\alpha S_n f'(S_n)\right]\right| \ \le \ \left(\sum_{i=1}^n \int_{-N}^N \left|\frac{\mcl K_\alpha(t,N)}n-\frac{ K_i(t,N)}{\alpha}\right| \dif t\right) \|f''\|+|\mcl R''|,
\end{equation*}
where $\|f''\|=\sup_{x \in \R} |f''(x)|$.
Therefore, in order to obtain the convergence rate, it suffices to bound $\|f''\|$ and the remainder $\mcl R''$.
 \vskip 3mm

 A recent result about stable convergence by Arras et. al. \cite[Appendix B]{AMPS16} is as the following: for $\alpha \in (1,2)$,
\ \ \ \
$$d_{{\rm Kol}} (\mcl L(S_{n}), \mu) \le C n^{-\frac 12(1-\frac \alpha 2)}, \ \ \ \ \ \ \ \ \sup_{h \in \mcl H_{3}} \left|\E[h(S_{n})]-\int_{\R} h(x) \mu(\dif x)\right| \le C n^{\frac{2\alpha}{2\alpha+1}(\frac 12-\frac{1}{\alpha})},$$
where $\mcl L(S_{n})$ is the distribution of $S_{n}$, $d_{{\rm Kol}}$ denotes the Kolmogorov distance, $\mu$ is a stable distribution with characteristic function $e^{-|\lambda|^{\alpha}}$, and $\mcl H_{3}$ is the set of all bounded third order differentiable functions $h$ such that $\|h^{{(k)}}\| \le 1$ for $k=0,1,2,3$. Their approach is by Stein-Tikhomirov method.
Note that \cite[Appendix B]{AMPS16} is a special case of Example 2 below, in which we show by our general result that a rate $n^{-\frac{2-\alpha}{\alpha}}$ in $W_{1}$ distance can be achieved. By a standard argument, this $W_1$ rate implies a Kolmogorov rate $n^{-\frac{2-\alpha}{2\alpha}}$, which is better than $n^{-\frac 12(1-\frac \alpha 2)}$.

More recently, Arras and Houdr\'e found a nice characterization of infinitely divisible law with finite first moment \cite[Theorem 3.1]{ArHo17}, and proved a general upper bound for $d_{\rm Kol}(\mu_{n},\mu)$ by Fourier analysis as $\mu_{n}$ and $\mu$ are both infinitely divisible. This result was applied to study several examples such as compound Poisson random variables, Pareto type random variables sum and so on, in particular, if $\mu_n$ is the distribution of a sum of i.i.d. infinitely divisible Pareto type random variables, it converges to a stable distribution with a rate $n^{-\frac{2-\alpha}{\alpha}}$ in Kolmogorov distance. They also derived a nice formulation of the related generators for the self-decomposable distribution family \cite[Proposition 5.1]{ArHo17}, which generalized the result in our Lemma \ref{l:Main2} below.  Furthermore, using a methodology very similar to the one developed in our paper, the same authors proved a bound for self-decomposable distribution approximation in a smooth Wasserstein distance $d_{W_{2}}$ by Stein's method, see \cite[Section 6]{ArHo17}.   Applying \cite[Theorems 6.1, 6.2]{ArHo17} to stable approximations, from the discrepancy terms in the bounds therein, we can immediately see that the convergence rate is at most $n^{-\frac{2-\alpha}{\alpha}}$ in $d_{W_{2}}$ distance. Note that $d_{W_{2}}$ is smaller than $W_{1}$ distance \cite[(4.3)]{ArHo17}.
\vskip 2mm

\cite{JoSa05} also gives a convergence rate for stable approximations in the Mallows distance $d_{r}$ with some $r>0$, note that $d_{r}$ is the classical Wasserstein-$r$ distance when $r \ge 1$. Let $X_{1},..., X_{n}$ be i.i.d. random variables with mean $0$ and a distribution function $F_X$ such that $F_{X}(x)=\frac{c_1+b_{X}(x)}{|x|^{\alpha}}$ for $x<0$ and $1-F_{X}(x)=\frac{c_{2}+b_{X}(x)}{|x|^{\alpha}}$ for $x>0$, where $c_{1}, c_{2}>0$ and $b_{X}(x)=O(\frac 1{|x|^{\gamma}})$ with $\gamma>0$, \cite[Theorem 1.2]{JoSa05} claims that $S_n=n^{-\frac 1\alpha} \sum_{i=1}^{n} X_{i}$ converges to a stable distribution $\mu$ with a rate $n^{\frac 1\beta-\frac 1\alpha}$ in the distance $d_\beta$ for some $\beta \in (\alpha,2]$. When $\alpha \in (1,2)$, this rate is worse than the rate $n^{-\frac 2 \alpha+1}$ in our paper, but $d_{\beta}$ is larger than $W_{1}$ distance. Moreover, when $\alpha \ge 1$ and $\gamma \ge 1$, one can take $\beta=2$ and thus gets a convergence rate $n^{\frac 12-\frac 1\alpha}$ in the Wasserstein-2 distance, which is not accessible by our  Stein's method. The theorem was proved by an idea from Lindeberg method and a coupling. More precisely, take a sequence of i.i.d. $\mu$-distributed random variables $Y_{1},...,Y_{n}$, since $n^{-1/\alpha} (Y_{1}+...+Y_{n})$ has the distribution $\mu$, it is easy to see that
\Bes
\begin{split}
d^{\beta}_{\beta} (\mcl L(S_{n}), \mu)& \ = \ n^{-\frac \beta \alpha} d^{\beta}_{\beta}\big(\sum_{i=1}^{n}X_{i}, \sum_{i=1}^{n}Y_{i}\big)  \ = \ n^{-\frac \beta \alpha} d^{\beta}_{\beta}\big(\sum_{i=1}^{n}X^{*}_{i}, \sum_{i=1}^{n}Y^{*}_{i}\big) \ \le \ n^{-\frac \beta \alpha}\E\big|\sum_{i=1}^{n}(X^{*}_{i}-Y^{*}_{i})\big|^{\beta},
\end{split}
\Ees
where $(X_{i}^*,Y_{i}^*)$ is a coupling of the distributions of $X_{i}$ and $Y_{i}$ \cite[(3)]{JoSa05}, which enjoys the property $\E |X^{*}_{i}-Y^{*}_{i}|^{\beta}=d^{\beta}_{\beta}(X_{i},Y_{i})$ for each $i$, and $\{(X_{i}^*,Y_{i}^*)\}_{1 \le i \le n}$ are independent. The previous relation, together with an inequality by von Bahr and Esseen \cite[(11), (12)]{JoSa05}, implies that $d^{\beta}_{\beta} (\mcl L(S_{n}), \mu)  \le  {2}{n^{-\beta/\alpha}} \sum_{i=1}^{n } d^{\beta}_{\beta}(X_{i},Y_{i})$. Since $d_{\beta}(X_{i},Y_{i})<\infty$ for some $\beta>\alpha$ \cite[Lemma 5.1]{JoSa05}, one immediately gets $d_{\beta} (\mcl L(S_{n}), \mu) = O(n^{\frac 1\beta-\frac 1\alpha})$.
\vskip 3mm

The organization of the paper is as follows. Section \ref{s:MThm} introduces notations and gives the two main theorems, while Section \ref{s:Ex} applies them to study three examples. The proofs of the two main theorems are given in Sections \ref{s:MThmProof} and \ref{s:MThm2Proof}  respectively, and the regularities of Stein's equation are proved in the 6th section. The last section is an appendix about the fourth example and some details of heat kernel estimates.

\vskip 3mm

{\bf Acknowledgements}: The author would like to gratefully thank editors and anonymous referees for very valuable corrections, suggestions and comments, which lead us to improve the paper. The author also would like to gratefully thank Zhen-Qing Chen, Elton Hsu, Tiefeng Jiang, Michel Ledoux, Ivan Nourdin, Gesine Reinert, Qi-Man Shao and  Ai-Hua Xia for very helpful discussions and comments. Special thanks are due to Rui Zhang, Xinghu Jin and Peng Chen for their going through the whole paper very carefully and giving numerous suggestions and corrections. This research is supported by the following grants:  Macao S.A.R. FDCT (038/2017/A1, 030/2016/A1, 025/2016/A1), NNSFC 11571390, University of Macau MYRG (2016-00025-FST, 2018-00133-FST).
\ \ \ \vskip 3mm

\section{Main results} \label{s:MThm}
Recall that $W_1$ distance between two probability measures $\mu_{1}$ and $\mu_{2}$ is defined by
\Be \label{e:DW12}
d_W(\mu_{1},\mu_{2})=\inf_{(X,Y) \in \mcl C(\mu_{1},\mu_{2})} \E |X-Y|,
\Ee
where $\mcl C(\mu_{1},\mu_{2})$ is the set of all the coupling realizations of $\mu_{1},\mu_{2}$. By a duality,
$$
d_W(\mu_{1},\mu_{2})\ =\ \sup_{h \in {\rm Lip}(1)} |\mu_{1}(h)-\mu_{2}(h)|,
$$
where ${\rm Lip}(1)=\{h: \R \rightarrow \R; \ |h(y)-h(x)| \le |y-x|\}$ and
$$\mu_{i}(h)=\int_{\R} h(x) \mu_{i}(\dif x), \ \ \ \ i=1,2.$$
Note that $d_W$ is also called $L^1$ distance. The Kolmogorov distance of $\mu_{1}$ and $\mu_{2}$ is defined by
$$d_{{\rm Kol}}(\mu_{1}, \mu_{2}):\ =\ \sup_{x \in \R} |\mu_1\left((-\infty,x]\right)-\mu_2\left((-\infty,x]\right)|.$$

For a sequence of measures $\{\nu_n\}_n$, we say they weakly converge to a measure $\nu$, denoted by $\nu_n \Rightarrow \nu$, if
$$\lim_{n \rightarrow \infty} \nu_{n}(f) \ = \ \nu(f)$$
for $f \in \mcl C_{b}(\R)$, all bounded continuous functions $f: \R \rightarrow \R$.
We use $C_{p}$ to denote some number which depends on parameter $p$, the exact value of $C_p$ may vary from line to line. We denote $\mcl L(X)$ the distribution of a given random variable $X$.
\ \ \ \vskip 3mm

Recall \eqref{e:Ker1-1} and \eqref{e:StaKer} in the introduction:
\Be \label{e:StaKer-1}
\mcl K_\alpha(t,N)=\frac{d_\alpha}{\alpha(\alpha-1)}\left(|t|^{1-\alpha}-N^{1-\alpha}\right),
\Ee
\Be  \label{e:Ker1-1-1}
 K_i(t,N)\ =\ \E\left[\zeta_{n,i} 1_{\{0 \le t \le \zeta_{n,i}  \le N\}}-\zeta_{n,i} 1_{\{-N \le  \zeta_{n,i}  \le t \le 0\}}\right],
\Ee
where $d_\alpha=\left(\int_{-\infty}^\infty \frac{1-\cos y}{|y|^{1+\alpha}} \dif y\right)^{-1}$ and $1 \le i \le n$. Note $d_{\alpha}=\frac{\alpha 2^{\alpha-1} \Gamma(\frac{1+\alpha}2)}{\sqrt{\pi} \Gamma(1-\frac \alpha2)}$ and $\lim_{\alpha \uparrow 2} \frac{d_{\alpha}}{2-\alpha}=1$, see \cite[p. 2800]{ChWa14}. Recall  the Gamma and Beta functions are respectively defined by
\Bes
\Gamma(x)=\int_{0}^{\infty} t^{x-1}e^{-t} \dif t, \ \ x>0;
\ \ \ \ \ \ {\rm B}(x,y)=\int_{0}^{1} t^{x-1}(1-t)^{y-1} \dif t, \ \ x>0, y>0.
\Ees

Let us now state our first main result, which is a general theorem giving a rate of stable law convergence in $W_1$ distance.
\begin{thm}  \label{t:MainThm} \footnote{Elton Hsu pointed out to the author that the condition '$S_{n} \Rightarrow \mu$' in Theorem \ref{t:MainThm} of the first draft can be removed.}
Let $n \in \N$ and let $\zeta_{n,1},...,\zeta_{n,n}$ be a sequence of independent random variables with $\E \zeta_{n,i}=0$ and $\E |\zeta_{n,i}|<\infty$ for $1 \le i \le n$.
Let $\mu$ be an $\alpha$-stable distribution with characteristic function $e^{-|\lambda|^\alpha}$
for $\alpha \in (1,2)$.  Then, we have
\ \ \
\Bes
\begin{split}
d_W\left(\mcl L (S_n), \mu\right) \ \le \ D_{\alpha}  \sum_{i=1}^n\int_{-N}^N \left|\frac{\mcl K_\alpha(t,N)}n -\frac{ K_i(t,N)}{\alpha}\right| \dif t \ +\ \mcl R_{N,n} \ \ \ \ \ \forall \ N>0,
\end{split}
\Ees
where $\mcl K_\alpha(t,N)$ and $K_{i}(t,N)$ are defined as above, $D_{\alpha}=\frac 4 \pi \sqrt{\frac{2\alpha+1}{\alpha}}{\rm B}\big(\frac{\alpha-1}\alpha,\frac 2\alpha\big) $,
\Bes
\begin{split}
\mcl R_{N,n}\ = \ 2  \sum_{i=1}^{n} \E\big(|\zeta_{n,i}|  1_{\{|\zeta_{n,i}|>N\}}\big)+\frac{4 d_{\alpha}}{\alpha-1}\frac{1}{N^{\alpha-1}}+\frac {D_{\alpha,\gamma}}n \sum_{i=1}^{n}\E|\zeta_{n,i}|^\gamma \ \ \ \ \ \forall \ \gamma \in (0,1),
\end{split}
\Ees
with $D_{\alpha,\gamma}=\frac{d_{\alpha}}{\alpha} \left[\frac{16}{\pi(2-\alpha)}\sqrt{\frac{\alpha+3}{\alpha}}+\frac{16}{\pi(\alpha-1)} \sqrt{\frac{2\alpha+1}{\alpha}}\right] {\rm B}\big(\frac{1-\gamma}{\alpha}, \frac{\gamma+\alpha}{\alpha}\big)$.
\end{thm}
\vskip 3mm
\begin{rem}
When $\alpha \le 1$, the stable distribution does not have its 1st moment, thus the corresponding $W_{1}$ is NOT well defined, see \eqref{e:DW12}. It is expected that $d_W(S_{n},\mu) \rightarrow \infty$ as $\alpha \downarrow 1$, this can be seen from
$$\lim_{\alpha \downarrow 1}D_{\alpha}=\infty, \ \ \ \lim_{\alpha \downarrow 1}D_{\alpha,\gamma}=\infty.$$ Moreover, $\lim_{\alpha \uparrow 2}D_{\alpha,\gamma}=\frac{2 \sqrt 5}{\pi} {\rm B}(\frac{1-\gamma}2, \frac{\gamma+2}{2})$ though there is a term $\frac{1}{2-\alpha}$ in $D_{\alpha,\gamma}$. Tables {\bf 1} and {\bf 2} give the values of $D_{\alpha}$ and $D_{\alpha,\gamma}$ respectively. Although $D_{\alpha,\gamma}$ is large, the term $\frac {D_{\alpha,\gamma}}n \sum_{i=1}^{n}\E|\zeta_{n,i}|^\gamma$ can be negligible in applications by taking $\gamma>2-\alpha$ and large $n$.
\end{rem}
\begin{rem}
Due to the lack of concentration phenomena of heavy tailed random variables sum, we can only observe the convergence after sampling a large number of random variables, see \cite[Section 5]{KuKe00} and Example 1 below. In applications, we take $\gamma=0.9$ so that the term $\frac {D_{\alpha,\gamma}}n \sum_{i=1}^{n}\E|\zeta_{n,i}|^\gamma$ will be small enough to be negligible as $n>10^{6}$.

\begin{table} \label{t:1}
\centering
\caption{The values of $D_{\alpha}$}
\begin{tabular}{c|ccccccccc}
\hline
    $\alpha$ &  1.1  & 1.2 & 1.3 & 1.4 & 1.5 & 1.6 & 1.7 & 1.8 & 1.9 \\
   \hline
    $D_{\alpha}$  & 22.14   & 11.45  & 8.04  & 6.42 & 5.51  & 4.94 & 4.57  & 4.32  & 4.15  \\
   \hline
\end{tabular}
\end{table}

\begin{table}  \label{t:2}
\centering
\caption{The values of $D_{\alpha,\gamma}$}
\begin{tabular}{c|ccccccccc}
\hline
&$\alpha=$ 1.1  & 1.2 & 1.3 & 1.4 & 1.5 & 1.6 & 1.7 & 1.8 & 1.9 \\
   \hline
$\gamma= $
0.1 & 33.13   & 19.01   &  14.40   &  12.17   & 10.89   & 10.09   &  9.55    & 9.18    &  8.91   \\
0.2 & 35.17   & 20.33   &  15.50   &  13.17   & 11.83   & 11.00   &  10.45   & 10.06   &  9.79   \\
0.3 & 38.59   & 22.43   &  17.17   &  14.64   & 13.20   & 12.30   &  11.70   & 11.30   &  11.01   \\
0.4 & 43.94   & 25.62   &  19.66   &  16.80   & 15.17   & 14.16   &  13.49   & 13.04   &  12.71   \\
0.5 & 52.30   & 30.53   &  23.45   &  20.05   & 18.12   & 16.92   &  16.13   & 15.59   &  15.21   \\
0.6 & 65.91   & 38.43   &  29.49   &  25.20   & 22.76   & 21.24   &  20.24   & 19.55   &  19.07   \\
0.7 & 90.04   & 52.33   &  40.06   &  34.16   & 30.79   & 28.69   &  27.31   & 26.36   &  25.68   \\
0.8 & 140.69  & 81.33   &  62.00   &  52.67   & 47.34   & 44.00   &  41.78   & 40.25   &  39.16   \\
0.9 & 298.18  & 171.06  &  129.58  &  109.52  & 98.02   & 90.78   &  85.95   & 82.58   &  80.16  \\
   \hline
\end{tabular}
\end{table}
\end{rem}

\begin{rem}   \label{r:Scaling}
If  $X$ has a stable distribution $\mu$ with characteristic function $e^{-|\lambda|^{\alpha}}$, then $\sigma^{1/\alpha} X$ has a distribution $\nu$ with characteristic function $e^{-\sigma |\lambda|^{\alpha}}$. By \eqref{e:DW12},
\Be
d_W(\mcl L(\sigma^{1/\alpha} S_{n}), \nu)\ =\ \sigma^{1/\alpha} d_W(\mcl L(S_{n}), \mu).
\Ee
On the other hand, it is easy to see from the definition of Kolmogorov distance that
$$
d_{{\rm Kol}}(\mcl L(\sigma^{1/\alpha} S_{n}), \nu)\ =\ d_{{\rm Kol}}(\mcl L(S_{n}), \mu).$$
\end{rem}



\vskip 3mm

Theorem \ref{t:MainThm} is a general theorem which bounds the $W_1$ distance of $\mcl L(S_{n})$ and $\mu$ by a discrepancy and a small remainder. An application of this theorem is to study the convergence rate of stable law. To this end, we first recall the classical stable law convergence theorem:
\vskip 2mm

 \begin{thm} [Theorem 3.7.2 of \cite{Dur10}]  \label{t:SLConvergence}
 Let $\xi_1,...,\xi_n,...$ be i.i.d. with a distribution that satisfies
 \ \ \ \
$$(i)\  \lim_{x \rightarrow \infty} \frac{\PP(\xi_1>x)}{\PP(|\xi_1|>x)}=\frac 12, \ \ \  \ \ \  \ \ \ (ii) \ \PP(|\xi_1|>x)=x^{-\alpha} L(x),$$
where $\alpha \in (0,2)$ and $L: [0, \infty) \rightarrow [0,\infty)$ is a slowly varying function, i.e. $\lim_{x \rightarrow \infty} \frac{L(tx)}{L(x)}=1$ for all $t>0$. Let $T_n=\xi_1+...+\xi_n$,
$A_n=\inf\{x: \PP(|\xi_1|>x) \le n^{-1}\}, \ B_n=n \E[\xi_1 1_{(|\xi_1| \le A_n)}].$ As $n \rightarrow \infty$, $(T_n-B_n)/A_n \Rightarrow \nu,$
 where $\nu$ is a symmetric stable distribution with characteristic function $\exp\left(-\frac{\alpha |\lambda|^{\alpha}}{2d_\alpha}\right)$. In particular, it follows from the property of stable distribution (see Remark \ref{r:Scaling}) that as $n \rightarrow \infty$,
\Be  \label{e:Scaling}
\left(\frac{\alpha}{2d_{\alpha}}\right)^{-\frac 1\alpha}\frac{T_{n}-B_{n}}{A_{n}}\Rightarrow \mu,
\Ee
where $\mu$ is a symmetric stable distribution with characteristic function $e^{-|\lambda|^{\alpha}}$.
\end{thm}
\vskip 3mm

In \cite[Theorem 3.7.2]{Dur10}, the limit of (i) is a general $c \in (0,1)$ rather than $\frac12$. When $c \ne \frac 12$, the limiting stable distribution $\nu$ is not symmetric. From the remark in \cite[p. 138]{Dur10}, we know that the conditions (i) and (ii) are also necessary for the above weak convergence to stable law.
Similar as studying a Berry-Esseen bound for a central limit theorem, we need to strengthen (i) and (ii) to get a rate for the convergence \eqref{e:Scaling}.

We assume that there exist some $A>0$ and two continuous functions $M_1: \R_+ \rightarrow \R$ and $M_2: \R_+ \rightarrow \R$, with $\lim_{x \rightarrow \infty} M_1(x)=0$ and $\lim_{x \rightarrow \infty} M_2(x)=0$, such that for all $x>A$,
\ \ \ \
$$(i') \ \ \frac{\PP(\xi_1>x)}{\PP(|\xi_1|>x)}\ =\ \frac{1+M_1(x)}2, \ \ \ \ \ \ \ (ii') \ \ \frac{\PP(|\xi_1|>x)}{\theta x^{-\alpha}}\ =\ 1+M_2(x),$$
where $\theta>0$ is a constant. We note that (i') and (ii') are equivalent to the condition that $\xi_{1}$ lies in the normal domain of attraction of $\mu$, which is generally stated as for all $x>A$,
\Be \label{e:NormAttract}
\PP(\xi_{1} > x)\ =\ c_{1} x^{-\alpha}(1+\delta_1(x)), \ \ \ \ \PP(\xi_{1}<-x)\ =\ c_{2} x^{-\alpha}(1+\delta_2(x)),
\Ee
where $c_{1}, c_{2} \ge 0$ with $c_{1}+c_{2}>0$ and $\lim_{x \rightarrow \infty} \delta_1(x)=0$ and $\lim_{x \rightarrow \infty} \delta_2(x)=0$, see \cite[p. 350]{Hall82} or \cite[Definition 5.1]{JoSa05}. If $\delta_1$ and $\delta_2$ both polynomially decay to $0$, then we call $\xi_1$ is in the \emph{strong} normal domain of attraction of $\mu$, see \cite[Definition 5.2]{JoSa05}. In our case, $c_{1}=c_{2}=\frac \theta 2$.
\vskip 3mm

Denote $\ell_{n}=\frac{\alpha \theta}{2d_{\alpha}} n$ and $b_{t}=\ell_{n}^{\frac 1\alpha}t+\E \xi_{1}$ for $t>0$ and
\ \ \ \ \
\Be  \label{e:Rt}
R_t\ = \ \frac 12 b_t^{-\alpha} (1+M_{1}(b_{t}))(1+M_{2}(b_{t})) \E \xi_{1},
\Ee
\Be  \label{e:rt}
r_{t}\ =\ \frac 12 b_{t}^{1-\alpha}\left[M_1+M_2+M_1 M_2\right] (b_t)+\frac 12\int_{b_t}^\infty s^{-\alpha} \left[M_1+M_2+M_1 M_2\right] (s) \dif s.
\Ee
Our second main theorem, which is essentially an application of Theorem \ref{t:MainThm}, is
\begin{thm}   \label{t:MThm2}
Let $\alpha \in (1,2)$, and let $\xi_1,...,\xi_n,...$ be i.i.d. with a distribution satisfying the conditions (i') and (ii'). Write $\zeta_{n,i}=\ell_{n}^{-\frac 1\alpha}(\xi_i-\E \xi_i)$ and $S_n=\zeta_{n,1}+...+\zeta_{n,n}$, then  $\mcl L(S_n) \Rightarrow \mu$ with characteristic function
$e^{-|\lambda|^{\alpha}}.$
Moreover, we have
\ \ \
\Bes
\begin{split}
d_W\left(\mcl L (S_n), \mu\right) \ & \le \ \frac{D_{\alpha}}{\alpha}\int_{-N}^{N} \big|\alpha \mcl K_{\alpha}(t,N)-n K_{1}(t,N)\big| \dif t+\mcl R_{N,n} \ \ \ \ \ \forall \ N>0,
\end{split}
\Ees
where
\ \ \ \
\Bes
\begin{split}
\mcl R_{N,n} & \ = \ D_{\alpha,\gamma} \ell^{-\frac{\gamma}{\alpha}}_{n} \E |\xi_{1}-\E \xi_{1}|^{\gamma} \\
& \ \ \ \ + \frac{4d_{\alpha}}{\delta_{n}^{\alpha-1}} \left(\frac{1+\delta_{n}^{\alpha-1}}{\alpha-1}+\frac{1}{\delta_{n}}+\frac{M_{2}(\ell_{n}^{\frac 1\alpha} N \delta_{n})}{\delta_{n}}+\int_{\delta_{n}}^\infty \frac{M_{2}(r \ell_{n}^{\frac 1\alpha} N)}{r^{\alpha} \delta^{1-\alpha}_{n}} \dif  r\right) N^{1-\alpha} \ \ \ \ \forall \ \gamma \in (0,1),
\end{split}
\Ees
with $\delta_{n}=1-\ell^{-\frac1 \alpha}_{n} N^{-1} |\E \xi_{1}|$. In particular, if  $\E \xi_{1}=0$, we have
\ \ \
\Be \label{e:RnNSym}
\begin{split}
\mcl R_{N,n} \ = \ D_{\alpha,\gamma}  \E |\xi_{1}|^{\gamma} \ell^{-\frac{\gamma}{\alpha}}_{n}
+4 d_{\alpha} \left(\frac{\alpha+1}{\alpha-1}+M_2(\ell^{\frac 1\alpha}_n N)+\int_{1}^{\infty} \frac{M_{2}(\ell_{n}^{\frac 1\alpha} N r)}{r^{\alpha}} \dif r\right)N^{1-\alpha}.
\end{split}
\Ee
\end{thm}
\vskip 3mm

It is worthy of stating the following corollary of Theorem \ref{t:MThm2}, from which we can fast determine the order of convergence rates.
\begin{corollary}  \label{c:MThm2}
Assume that the same conditions as in Theorem \ref{t:MThm2} hold. We have
\ \ \ \ \
\Bes
d_W\left(\mcl L (S_n), \mu\right) \ \le \ C_{\alpha} \left(n^{-\frac{2-\alpha}{\alpha}}+N^{1-\alpha}+n^{1-\frac 1\alpha}N |r_{N}|+n^{1-\frac 1\alpha} \int_{4(A+|\E \xi_{1}|) \ell^{-\frac 1\alpha}_{n} \le |t| \le N} |r_{t}| \dif t\right).
\Ees
\end{corollary}
\vskip 3mm
We end this section with the following lemma, which will be used from time to time later.
\begin{lem} \label{l:MomEst}
Let $X$ be a random variable, for any $t>0$ we have
\Be
\E\left[X 1_{\{X>t\}}\right]=t \PP(X>t)+\int_t^\infty \PP(X>r)\dif r.
\Ee
\end{lem}
\begin{proof}
Observe by Fubini's Theorem that
\Be
\begin{split}
\E\left[X 1_{\{X>t\}}\right]\ &=\ \int_0^\infty \E \left[1_{\{0 \le r < X\}} 1_{\{X>t\}}\right] \dif r \\
&\ =\ \int_0^t \E \left[1_{\{X>t\}}\right] \dif r+\int_t^\infty \E \left[1_{\{X>r\}}\right] \dif r,
\end{split}
\Ee
from which we immediately obtain the inequality in the lemma, as desired.
\end{proof}

\section{Three examples} \label{s:Ex}

In this section, we shall use our results to study three examples which have been considered \cite{DaNa02, KuKe00, Hall81, AMPS16}, the known literatures only gave the order of convergence rates in Kolmogorov distance. In contrast, using our Theorems \ref{t:MainThm} and \ref{t:MThm2}, we can obtain explicit bounds for these examples in $W_{1}$ distance, and fast determine the order of the convergence rates by Corollary \ref{c:MThm2}. In the regime $\alpha \in (1,2)$, most of our results are as good as or better than the known ones.

In the appendix, we further consider the fourth example which is out of the scope of normal domain of attraction of stable law. A related example was studied in \cite{JuPa98} and the convergence rate in Kolmogorov distance is $(\log n)^{-1}$. By our results, we obtain a rate $(\log n)^{-1+\frac 1\alpha}$. Because the calculation is very complicated and long, we will not give an explicit bound but only figure out its leading order in the appendix.   \\

\noindent {\bf Example 1: Pareto distribution case \cite{DaNa02, KuKe00}}

Assume that $\xi_1,...,\xi_n,...$ be i.i.d. with a Pareto distribution with $\alpha \in (1,2)$, i.e.,
\ \ \ \
$$\PP\left(\xi_1 \ge x\right)=\frac{1}{2 |x|^{\alpha}}, \ \ x \ge 1,\ \ \ \ \ \ \ \ \ \PP\left(\xi_1 \le x\right)=\frac{1}{2 |x|^{\alpha}}, \ \ x \le -1,$$
i.e., $\xi_1$ has a density function $p(x)$:
$$p(x)=0, \ \ \ |x| \le 1; \ \ \ \ \ \ \ \ \ \ \ p(x)=\frac{\alpha}{2|x|^{\alpha+1}}, \ \ \ \ |x|>1. $$
By Theorem \ref{t:SLConvergence}, we have $B_n=0$ and $A_n= n^{1/\alpha}$. Denote $\ell_{n}=\frac{\alpha}{2 d_{\alpha}} n$ and
$$\zeta_{n,i}\ =\ \ell_{n}^{-\frac 1\alpha} \xi_i, \ \ \ \ \ \ \ \ i=1,...,n,$$
 $S_n$ weakly converges to a stable distribution $\mu$ with characteristic function $e^{-|\lambda|^{\alpha}}$. We can directly apply Theorem \ref{t:MThm2} to get a convergence rate $n^{-\frac{2-\alpha}{\alpha}}$, but it is very instructive to prove this rate by applying Theorem \ref{t:MainThm} directly.
\vskip 2.5mm

It is straightforward to check that the terms in $\mcl R_{N,n}$ are
\ \ \
 \Bes
 \begin{split}
& 2  \sum_{i=1}^{n} \E\big(|\zeta_{n,i}|  1_{\{|\zeta_{n,i}|>N\}}\big) \ = \ \frac{4 d_{\alpha}}{\alpha-1} N^{1-\alpha}, \\
& \frac {D_{\alpha,\gamma}}n \sum_{i=1}^{n}\E|\zeta_{n,i}|^{\gamma}\ = \  \frac{\alpha D_{\alpha, \gamma}}{\alpha-\gamma} \left(\frac{2 d_{\alpha}}{\alpha}\right)^{\frac{\gamma}{\alpha}}  n^{-\frac{\gamma}{\alpha}},
 \end{split}
 \Ees
 thus
$$\mcl R_{N,n} \ = \ \frac{8 d_{\alpha}}{\alpha-1} N^{1-\alpha}+\frac{\alpha D_{\alpha, \gamma}}{\alpha-\gamma} \left(\frac{2 d_{\alpha}}{\alpha}\right)^{\frac{\gamma}{\alpha}}  n^{-\frac{\gamma}{\alpha}}.$$
It remains to compute the integral term in the bound of Theorem \ref{t:MainThm}. Recall \eqref{e:Ker1-1-1}, when $t \ge 0$,
\Be
\begin{split}
 K_1(t,N) & \ = \ \E\left[\zeta_{n,1} 1_{\{0 \le t \le \zeta_{n,1}  \le N\}}\right] \\
&\ =\ \ell_n^{-\frac{1}{\alpha}} \int_{\ell_n^{1/\alpha} t}^{\ell_n^{1/\alpha} N} x p(x) \dif x \\
&\ =\  \frac{\alpha}{2\ell_n(\alpha-1)} \left[\left(t \vee \frac{1}{\ell_n^{1/\alpha}}\right)^{-\alpha+1} -\frac 1{N^{\alpha-1}}\right].
\end{split}
\Ee
By the symmetry property of $p(x)$, we have
\Be
\begin{split}
 K_1(t,N)&\ =\  \frac{\alpha}{2\ell_n(\alpha-1)} \left[\left(|t| \vee \frac{1}{\ell_n^{1/\alpha}}\right)^{-\alpha+1} -\frac 1{N^{\alpha-1}}\right], \ \ \ \ \ t \le 0.
\end{split}
\Ee
Hence,
\Be
\begin{split}
& \ \ \ \ \ \ \sum_{i=1}^n\int_{-N}^N \left|\frac{\mcl K_\alpha(t,N)}n -\frac{ K_i(t,N)}{\alpha   }\right| \dif t \\
& \ = \  \int_{-N}^N \left|\frac{d_\alpha}{\alpha (\alpha-1)} \left(\frac{1}{|t|^{\alpha-1}}-\frac{1}{N^{\alpha-1}}\right)-\frac{n K_1(t,N)}{\alpha   }\right|\dif t \\
&\ =\  \frac{d_\alpha}{\alpha (\alpha-1)} \int_{-N}^N \left|\frac{1}{|t|^{\alpha-1}}-\left(|t| \vee \frac{1}{\ell_n^{1/\alpha}}\right)^{-\alpha+1}\right|\dif t \\
& \ = \ \frac{1}{2-\alpha}\left(\frac{2 d_{\alpha}}{\alpha}\right)^{\frac 2 \alpha }  n^{-\frac{2-\alpha}{\alpha}}.
\end{split}
\Ee
So, we have
\ \  \
\Be
d_W(\mcl L(S_n), \mu) \ \le \ \frac{8 d_{\alpha}}{\alpha-1} N^{1-\alpha}+\frac{D_{\alpha}}{2-\alpha}\left(\frac{2 d_{\alpha}}{\alpha}\right)^{\frac 2 \alpha }  n^{-\frac{2-\alpha}{\alpha}}+\frac{\alpha D_{\alpha, \gamma}}{\alpha-\gamma} \left(\frac{2 d_{\alpha}}{\alpha}\right)^{\frac{\gamma}{\alpha}}  n^{-\frac{\gamma}{\alpha}}.
\Ee
Since $N$ is arbitrary, let $N \rightarrow \infty$, we get
\Be  \label{e:dEx1}
d_W(\mcl L(S_n), \mu) \ \le \ \frac{D_{\alpha}}{2-\alpha} \left(\frac{2 d_{\alpha}}{\alpha}\right)^{\frac 2 \alpha }  n^{-\frac{2-\alpha}{\alpha}}+\frac{\alpha D_{\alpha, \gamma}}{\alpha-\gamma} \left(\frac{2 d_{\alpha}}{\alpha}\right)^{\frac{\gamma}{\alpha}}  n^{-\frac{\gamma}{\alpha}} \ \ \ \ \ \ \ \forall \ \gamma \in (0,1).
\Ee
\vskip 3mm

Let us compare our result with the known results in literatures.
The reference \cite{KuKe00} gave a convergence rate:
$$d_{\rm Kol}(\mcl L(S_n), \mu) \ \le \  C_\alpha\begin{cases} n^{-\frac{2-\alpha}\alpha}, \ \ \ \ \ & \alpha \in (1,2), \\
n^{-1}, \ \ \ & \alpha \in (0,1],
\end{cases}
$$
where an exact value of $C_{\alpha}$ was not given.

When $\alpha \in (1,2)$, the authors of \cite{DaNa02} obtained a rate $n^{-\frac{2-\alpha}{d+\alpha}}$ for $d$ dimensional stable law in total variation distance and conjectured that the rate can be improved to $n^{-\frac{2-\alpha}{\alpha}}$ in $L^{1}$ or total variation distance. Our result gives a positive answer to their conjecture for the $L^{1}$ distance case when $d=1$.
\vskip 3mm

Table {\bf 3} gives exact bounds of $d_{W}(\mcl L(S_n), \mu)$ as $n=10^{6}$ according to \eqref{e:dEx1}, which vary according to $\alpha$ and $\gamma$.  Due to the lack of concentration phenomena in heavy detailed random variables sum, in simulations one has to take large samples (often more than $10^{6}$) to observe the convergence.  \cite[Section 5]{KuKe00} only simulated the limiting behavior of $\E|S_{n}|$ for $\alpha=1.5$, the convergence can be well observed only after the sample size reaches $10^{6}$.

\begin{table}  \label{t5-5}
\centering
\caption{Exact bounds of $d_{W}(\mathcal{L}(S_{n}),\mu)$ with $n=10^{6}$}
\begin{tabular}{c|ccccccccc}
\hline
 &$ \alpha=$ 1.1  & 1.2 & 1.3 & 1.4 & 1.5 & 1.6 & 1.7 & 1.8 & 1.9 \\
   \hline
$\gamma= $
0.1 & 9.906  &  6.245  &  5.121 &   4.636 &   4.424 &   4.399 &    4.588 &    5.114 &    6.174  \\
0.2 & 3.176  &  2.213  &  1.975 &   1.925 &   1.970 &   2.112 &    2.418 &    3.030 &    4.154 \\
0.3 & 1.066  &  0.818  &  0.792 &   0.833 &   0.921 &   1.087 &    1.407 &    2.032 &    3.177 \\
0.4 & 0.377  &  0.317  &  0.333 &   0.380 &   0.462 &   0.617 &    0.926 &    1.544 &    2.694 \\
0.5 & 0.142  &  0.131  &  0.149 &   0.186 &   0.255 &   0.396 &    0.692 &    1.300 &    2.451 \\
0.6 & 0.059  &  0.058  &  0.073 &   0.101 &   0.160 &   0.289 &    0.576 &    1.177 &    2.327 \\
0.7 & 0.027  &  0.029  &  0.040 &   0.063 &   0.115 &   0.238 &    0.518 &    1.114 &    2.263 \\
0.8 & 0.016  &  0.018  &  0.026 &   0.046 &   0.095 &   0.214 &    0.490 &    1.084 &    2.232  \\
0.9 & 0.014  &  0.015  &  0.023 &   0.042 &   0.091 &   0.210 &    0.487 &    1.081 &    2.230 \\
 \hline
\end{tabular}
\end{table}
\ \\
\noindent {\bf Example 2: Convergence rate of Pareto densities with modified tails (\cite[Section 3]{KuKe00}, \cite[Appendix B]{AMPS16})}
\vskip 3mm
 In \cite[Section 3]{KuKe00}, a sequence of i.i.d random variables $(\xi_{n})_{n \ge 1}$ with the following density were considered:
\Be
p(x)\ =\ \frac{A}{|x|^{1+\alpha}}+\frac{B}{|x|^{1+\beta}} \ \ {\rm for} \ \ |x|>a,  \ \ \ \ \ \ \ p(x)\ =\ 0 \ \ {\rm for} \ \ |x| \le a,
\Ee
where $0<\alpha<2$, $\alpha<\beta$, $A>0$ and $B>0$.
When $\beta>2 \alpha$, it was proved that $n^{-1/\alpha} \sum_{i=1}^{n} \xi_{i}$ converges to a stable distribution $\nu$ in Kolmogorov distance with a rate $n^{-\frac{2-\alpha}{\alpha}}$ for $\alpha \in (1,2)$ and a rate $n^{-1}$ for $\alpha \in (0,1]$. When $\beta \in (\alpha, 2 \alpha)$, the rate is $n^{-\min\left(\frac{\beta}{\alpha}-1, \frac{2-\alpha}{\alpha}\right)}$ for $\beta \ne 2$ and $n^{-\frac{2-\alpha}{\alpha}} \log n$ for $\beta=2$, See \cite[(3.6), Table 1]{KuKe00}.
\vskip 3mm

We now determine an explicit bound for $d_{W}(\mcl L(S_{n}), \mu)$ by Theorem \ref{t:MThm2}. Without loss of generality, we assume $a=1$ (otherwise take $\tl \xi_{i}=a^{-1} \xi_{i}$) and thus have
\Be \label{e:SpEx2}
p(x)\ =\ \frac{A}{2|x|^{1+\alpha}}+\frac{B}{2|x|^{1+\beta}} \ \ {\rm for} \ \ |x|>1,  \ \ \ \ \ \ \ p(x)\ =\ 0 \ \ {\rm for} \ \ |x| \le 1,
\Ee
with $\frac A{\alpha}+\frac B{\beta}=1$ and $\beta \in (\alpha,\infty)$. We can easily determine
$\theta=\frac{A}{\alpha}$, $\ell_{n}=\frac{A}{2d_{\alpha}} n$ and $S_{n}=\ell_{n}^{-\frac 1\alpha} \sum_{i=1}^{n} \xi_{i}$
 and
$$M_{1}(x)\ =\ 0, \ \ \ \ M_{2}(x)\ =\ (B\alpha)(A\beta)^{-1} x^{-(\beta-\alpha)}, \ \ \ \ \ \ \ x \ge1.$$
To use Theorem \ref{t:MThm2}, we need to compute $n \mcl K_\alpha (t,N)$ and $\mcl R_{N,n}$ therein.
By a straightforward calculation, we have
\Bes
\begin{split}
n K_1(t,N) 
& \ =\ \frac{d_{\alpha}}{\alpha-1} \left(\left(|t|\vee \ell_{n}^{-\frac 1\alpha}\right)^{1-\alpha}-N^{1-\alpha}\right)+\frac{Bn\ell_{n}^{-\frac{\beta}{\alpha}}}{2(\beta-1)}  \left(\left(|t|\vee \ell_{n}^{-\frac 1\alpha}\right)^{1-\beta}-N^{1-\beta}\right).
\end{split}
\Ees
By a similar computation as in Example 1, when $\beta \ne 2$,
\ \ \
\Bes
\begin{split}
\int_{-N}^{N} \left|\alpha \mcl K_{\alpha}(t,N)-n K_1(t,N)\right| \dif t & \ = \
\frac{2d_{\alpha} \ell_{n}^{\frac{\alpha-2}{\alpha}}}{2-\alpha}+\frac{2 A d_{\alpha}}{B(\beta-2)} \left(\ell_{n}^{\frac {\alpha-2}\alpha}-\ell_{n}^{\frac{\alpha-\beta}{\alpha}} N^{2-\beta}\right);
\end{split}
\Ees
when $\beta = 2$,
\ \ \
\Bes
\begin{split}
\int_{-N}^{N} \left|\alpha \mcl K_\alpha(t,N)-n K_1(t,N)\right| \dif t & \ = \
\frac{2d_{\alpha} \ell_{n}^{\frac{\alpha-2}{\alpha}}}{2-\alpha}+\frac{2A d_{\alpha}}B \ell_{n}^{\frac{\alpha-2}\alpha}\left[\log N+\frac 1\alpha \log \ell_{n}\right].
\end{split}
\Ees
By \eqref{e:RnNSym}, we immediately obtain
\ \ \
\Be
\mcl R_{N,n} \ = \ D_{\alpha,\gamma}\left(\frac{A}{\alpha-\gamma}+\frac{B}{\beta-\gamma}\right) \ell_{n}^{-\frac{\gamma}{\alpha}}+4 d_{\alpha}\left[\frac{\alpha+1}{\alpha-1}+\frac{B\alpha}{A(\beta-1)} \ell_{n}^{\frac{\alpha-\beta}{\alpha}} N^{\alpha-\beta}\right] N^{1-\alpha}.
\Ee
\vskip 3mm

\noindent Combining the previous relations, we immediately obtain an explicit bound for $d_{W}(\mcl L(S_{n}),\mu)$, which has a leading term and a remainder $\mcl R(n)$, both having explicit values. More precisely, (note $\ell_{n}=\frac{A}{2d_{\alpha}} n$), we have
\vskip 3mm
\noindent (1).  When $\beta>2$, take $N \rightarrow \infty$,
\Bes  
d_{W}(\mcl L(S_{n}), \mu) \ \le \ \frac{2 d_{\alpha} D_{\alpha}}{\alpha} \left(\frac{1}{2-\alpha}+\frac{A}{B(\beta-2)}\right) \ell^{-\frac{2-\alpha}\alpha}_{n}+\mcl R(n)
\Ees
with
$$\mcl R(n) \ = \  D_{\alpha, \gamma} \left(\frac{A}{\alpha-\gamma}+\frac{B}{\beta-\gamma}\right) \ell^{-\frac\gamma \alpha}_{n}.$$
\vskip 2mm

\noindent (2). When $\beta=2$, take $N=\ell_{n}^{q}$ with $q \ge \frac{2-\alpha}{\alpha(\alpha-1)}$ being arbitrary,
\Bes
d_{W}(\mcl L(S_{n}), \mu)  \ \le \ \frac{2 d_{\alpha}D_{\alpha} A(\alpha q+1)}{\alpha^{2} B} \ell_{n}^{-\frac{2-\alpha}\alpha}  \log \ell_{n}+\mcl R(n)
\Ees
with
\Bes
\begin{split}
\mcl R(n)\ =\ \frac{2 d_{\alpha }D_{\alpha}}{(2-\alpha)\alpha}  \ell^{-\frac{2-\alpha}\alpha}_{n}&+D_{\alpha, \gamma} \left(\frac{A}{\alpha-\gamma}+\frac{B}{2-\gamma}\right) \ell^{-\frac\gamma \alpha}_{n}+4 d_{\alpha}\left(\frac{\alpha+1}{\alpha-1} \ell_{n}^{-(\alpha-1) q}
+\frac{\alpha B}A \ell_{n}^{-\frac{2-\alpha}\alpha-q}\right).
\end{split}
\Ees
\vskip 2mm

\noindent (3). When $\alpha<\beta<2$, take $N=\ell_{n}^{q}$ with $q=\frac {\beta-\alpha}{\alpha(\alpha+1-\beta)}$ being arbitrary,
\Bes
\begin{split}
d_{W}(\mcl L(S_{n}), \mu) & \ \le \ 2d_{\alpha} \left[\frac{D_{\alpha}}{(2-\beta)\alpha}+\frac{2(\alpha+1)}{\alpha-1}\right] \ell_{n}^{-\frac {(\beta-\alpha)(\alpha-1)}{\alpha(\alpha+1-\beta)}}+\mcl R(n)
\end{split}
\Ees
with
\ \
\Bes
\begin{split}
\mcl R(n)\ = \ 2 d_{\alpha}\left(\frac{1}{2-\alpha}-\frac{A}{(2-\beta)B}\right) & \ell^{-\frac{2-\alpha}\alpha}_{n}+D_{\alpha, \gamma} \left(\frac{A}{\alpha-\gamma}+\frac{B}{2-\gamma}\right) \ell^{-\frac\gamma \alpha}_{n}+\frac{4 A\alpha d_{\alpha}}{B(\beta-1)} \ell_{n}^{-\frac {2\alpha+2+\alpha \beta-3\beta-\alpha^{2}}{\alpha(\alpha+1-\beta)}} .
\end{split}
\Ees

\vskip 3mm
Note that the case (1) covers the example considered in \cite[Appendix B]{AMPS16}, in which $\beta=1+\alpha$. By a standard argument, the bound in (1) implies
$$d_{{\rm Kol}}(\mcl L(S_n), \mu) \ \le \  C n^{-\frac{2-\alpha}{2\alpha}},$$
which is better than the rate $n^{-\frac 12(1-\frac \alpha 2)}$ in \cite{AMPS16}.
\vskip 3mm

We can consider a more general distribution:
\Be  \label{e:GenExa2}
\begin{split}
& \PP\left(\xi_{1} > x\right) \ = \ \frac{A}{2|x|^{\alpha}}+\frac{B_{1}(x)}{2|x|^{\beta}}, \ \ \ \ \ x>a, \\
& \PP\left(\xi_{1}< x\right) \ = \ \frac{A}{2|x|^{\alpha}}+\frac{B_{2}(x)}{2|x|^{\beta}}, \ \ \ \ \ x<-a, \\
\end{split}
\Ee
where $\alpha \in (1,2)$, $\alpha<\beta$, $a>0$, $B_{1}(x)$ and $B_{2}(x)$ are both continuous functions such that $-L \le B_{1}(x), B_{2}(x) \le L$ for all $x \in \R$ and some constant $L>0$. Take $\theta=A$ and $\ell_{n}=\frac{\theta \alpha}{2d_{\alpha}}n$, we have
$$S_{n} \ \Rightarrow \ \mu \ \ \ \ \ {\rm with} \ \ \ \ S_{n}=\ell_{n}^{-\frac 1\alpha} \sum_{i=1}^{n} (\xi_{i}-\E \xi_{i}).$$
By Theorem \ref{t:MThm2}, we can obtain an explicit bound for $d_{W}(\mcl L(S_{n}),\mu)$ by a similar but much more complicated calculation. Here, we would like to omit the detailed calculation but get the order of the rate. More precisely, by Corollary \ref{c:MThm2} we have

\ \ \ \ \
\Bes
d_W(\mcl L(S_n), \mu) \ \le \ C \left(n^{\frac{\alpha-2}\alpha}+N^{1-\alpha}+N^{2-\beta} n^{1-\frac{\beta}{\alpha}}\right), \ \ \ \ \ \ \beta  \ne 2,
\Ees
\Bes
d_W(\mcl L(S_n), \mu) \ \le \ C  \left(n^{\frac{\alpha-2}\alpha}+N^{1-\alpha}+n^{1-\frac{2}{\alpha}} \log n+n^{1-\frac{2}{\alpha}} \log N\right), \ \ \ \ \ \ \beta=2,
\Ees
where $C$ depends on $\alpha,\beta,a,L,A,B$. Hence,

\noindent (i). When $\beta > 2$, let $N \rightarrow \infty$, we get
$$d_{{\rm W}}(\mcl L(S_n), \mu) \ \le \  C \ n^{-\frac{2-\alpha}\alpha}.$$
(ii). When $\beta = 2$, taking $N=n^{\frac{2-\alpha}{\alpha(\alpha-1)}}$, we get
$$d_{{\rm W}}(\mcl L(S_n), \mu) \ \le \  C \ n^{-\frac{2-\alpha}\alpha} \log n.$$
(iii). When $\alpha< \beta<2$, taking $N=n^{\frac{\beta-\alpha}{\alpha(1+\alpha-\beta)}}$, we have
 $$d_W(\mcl L(S_n), \mu) \ \le\  C \ n^{-\frac{(\alpha-1)(\beta-\alpha)}{\alpha(1+\alpha-\beta)}}.$$
\vskip 3mm

As seen from the results in the previous two examples, the $\gamma$ in the bounds of $d_{W}(\mcl L(S_{n}),\mu)$ may vary from $0$ to $1$, its optimal choice depends on $\alpha$ and $n$. When $n=10^6$, we plot Figure {\bf 1} to demonstrate the optimal $\gamma$ as a function of $\alpha$ for $\alpha=1+j/100$ with $j\in\{1,\cdots,99\}$ for the following four cases: (1). Example 1 (red line);  (2). Example 2 with $\beta=4$ and $A=B=\frac{\alpha \beta}{\alpha+\beta}$ (green line); (3). Example 2 with $A=B$ and $\beta=2$ (blue line); (4). Example 2 with $A=B$ and  $\beta=\alpha+0.1$ (black line).
\begin{figure}\label{beta=4}
\centering
\includegraphics[height=8cm,width=12cm]{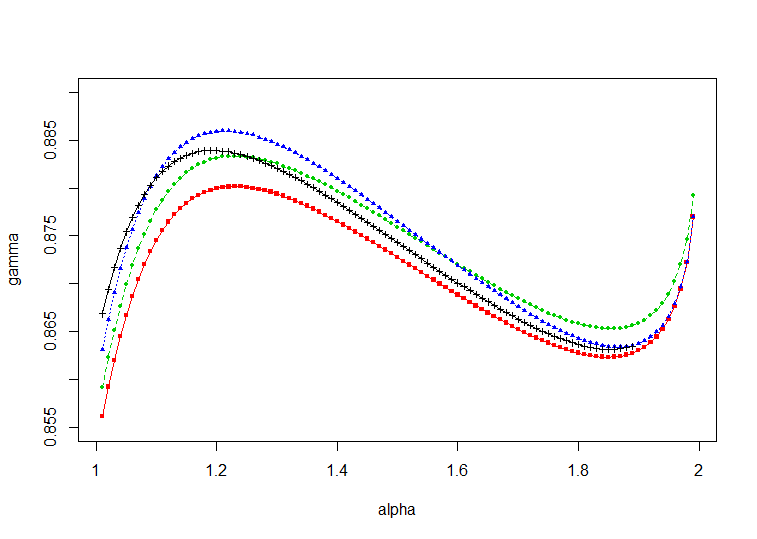}
\caption{the optimal choice of $\gamma$ with $n=10^{6}$}
\end{figure}
\\

\noindent {\bf Example 3: An example of Hall \cite{Hall81}}

Let $Z_1,...,Z_n,...$ be a sequence of i.i.d. random variables such that $Z_1$ has a density function $f(x)$ in the interval $-\e<x<\e$ for some $\e>0$. For further use, we denote $f(x)=a_0+h_1(x)$ for $x \in [0, \e)$ and $f(x)=b_0+h_2(|x|)$ for $x \in (-\e,0]$ where $a_0>0$, $b_0>0$, $h_1, h_2$ are both continuous positive functions from $[0, \e)$ to $\R^+$. Let $X_i={\rm sgn}(Z_i) |Z_i|^{-\frac 1\alpha}$, Hall studied the convergence rate of the following sum
\Be  \label{e:HallSum}
n^{-\frac 1\alpha} \left[\sum_{i=1}^n X_{i}-k_n(\alpha)\right],
\Ee
where $k_n(\alpha)$ is some number and $0<\alpha<2$, he proved
\ \ \ \
\begin{thm} [Theorem 3 of \cite{Hall81}]
Suppose $1<\alpha<2$, $f(x)=a_0+h_{1}(x)$ for $x \in [0,\e)$ and $f(x)=b_0+h_{2}(|x|)$ for $x \in (-\e,0]$ where $|h_1(x)|+|h_{2}(x)| \le b |x|^c$ for some $c>0$ with $\alpha(c+1)<2$, i.e., $0<c<\frac{2-\alpha}{\alpha}$, then $n^{-\frac 1\alpha}\left(\sum_{i=1}^n X_i-n \E X_1\right)$ weakly converges to a stable distribution with the distribution function $A(x)$. Moreover,
\Be
\sup_{x \in \R} \left|\PP \left[n^{-\frac 1\alpha}\left(\sum_{i=1}^n X_i-n \E X_1\right) \le x \right]-A(x)\right|=O(n^{-c}).
\Ee
\end{thm}

As an application of the case \eqref{e:SpEx2} above, we can study a special case of \eqref{e:HallSum} and give an explicit bound of the convergence in $W_{1}$ distance. More precisely, let $\alpha \in (1,2)$, we assume $a_0=b_0=a$ and $h_{1}(x)=h_{2}(x)=b x^{c}$ for $[0,1]$, by a straightforward calculation,
\ \ \
\Bes
\begin{split}
& \PP \left(X_1>x\right) \ = \ \frac {2a} {2x^{\alpha}}+\frac{\frac{2b}{c+1}}{2 x^{\alpha(c+1)}}, \ \ \ \ \ \ \ \ x>1; \\
&\PP \left(X_1<x\right) \ = \ \frac{2a}{2|x|^{\alpha}}+\frac{\frac{2b}{c+1}}{2 |x|^{\alpha(c+1)}}, \ \ \ \ \ \ x<-1.
\end{split}
\Ees
We have the following theorem about explicit bound of convergence rate in $W_{1}$ distance.
\begin{thm} \label{t:Hall}
Let the above assumptions hold.
Take $A=2a$, $B=\frac{2b}{c+1}$, $\beta=\alpha(c+1)$, $\ell_{n}=\frac{A \alpha}{2d_{\alpha}} n$ and $S_{n}=\ell_{n}^{-\frac 1\alpha} \sum_{i=1}^{n} \xi_{i}$, we have
$S_{n} \Rightarrow \mu$ with $\mu$ a stable distribution having characteristic function $e^{-|\lambda|^{\alpha}}$. Moreover, $d_{W}(\mcl L(S_{n}),\mu)$ has an \emph{explicit} bound the same as the cases (1)-(3) in Example 2. More precisely,
\begin{itemize}
\item If $\beta>2$, i.e., $c>\frac{2-\alpha}{\alpha}$, the bound of the case (1) holds with leading order $n^{-\frac{2-\alpha}{\alpha}}$.
\item If $\beta=2$, i.e., $c=\frac{2-\alpha}{\alpha}$, the bound of the case (2) holds with leading order $n^{-\frac{2-\alpha}{\alpha}} \log n$.
\item If $\alpha<\beta<2$, i.e., $c \in (0, \frac{2-\alpha}{\alpha})$, the bound of the case (3) holds with leading order $n^{-\frac{(\alpha-1)c}{1-\alpha c}}$.
\end{itemize}
\end{thm}

We can give the exact values of coefficients in Theorem \ref{t:Hall} and those of $d_W\left(\mcl L(S_n), \mu\right)$, the results are the same as those in Example 2, thus we omit them.
Similar as the case \eqref{e:GenExa2} in the previous example, we can apply Theorem \ref{t:MThm2} to more general distributions and get an explicit bound for the corresponding $d_{W}(\mcl L(S_{n}),\mu)$ by much more complicated calculations. Here we only use Corollary \ref{c:MThm2} to give the order of the convergence rate as the following.
\ \ \ \ \ \
\begin{thm}
Suppose $1<\alpha<2$, $a_{0}=b_{0}=a$ and $|h_1(x)|+|h_{2}(x)| \le b |x|^c$ for $x \in [0, \e)$ with $b>0,  \e>0, c>0$, let $\ell_{n}=\frac{a \alpha}{d_{\alpha}}n$ then $S_n:=\ell_{n}^{-\frac 1\alpha}(\sum_{i=1}^n X_i-n \E X_1)$ weakly converges to a stable distribution $\mu$ with characteristic function $e^{-|\lambda|^{\alpha}}$. Moreover,
\Be
d_W\left(\mcl L(S_n), \mu\right)\ \le \ C \begin{cases}
n^{-\frac{2-\alpha}{\alpha}},   \ \ \ & c > \frac{2-\alpha}{\alpha} \\
n^{-\frac{2-\alpha}{\alpha}} \log n, \ \ \ &c=\frac{2-\alpha}{\alpha}; \\
n^{-\frac{(\alpha-1)c}{1-\alpha c}}, \ \ \ & 0<c<\frac{2-\alpha}{\alpha}.
\end{cases}
\Ee
\end{thm}

\ \ \ \

\section{Proof of Theorem \ref{t:MainThm}: Stein's method}  \label{s:MThmProof}

\subsection{Stein's equation and its regularity estimates}
Let us first recall the definition of fractional Laplacian $\Delta^{\frac \alpha 2}$ with $\alpha \in (0,2)$, see for instance \cite{Kw17}.
Let $f:\R \rightarrow \R$ be a measurable function, for any $x \in \R$, $\Delta^{\frac \alpha 2} f(x)$ is defined by
\begin{equation}  \label{e:FraLap0}
\Delta^{\frac \alpha 2} f(x)\ = \ d_\alpha \left({\rm p.v.} \int_\R \frac{f(x+y)-f(x)}{|y|^{1+\alpha}} \dif y\right),
\end{equation}
provided the principal value ${\rm p.v.} \int_\R \frac{f(x+y)-f(x)}{|y|^{1+\alpha}} \dif y$ exists, where $d_\alpha\ =\ \left(\int_{-\infty}^\infty \frac{1-\cos y}{|y|^{1+\alpha}} \dif y\right)^{-1}$ and
$${\rm p.v.} \int_\R \frac{f(x+y)-f(x)}{|y|^{1+\alpha}} \dif y \ = \ \lim_{r\downarrow 0} \int_{\R \setminus (-r,r)} \frac{f(x+y)-f(x)}{|y|^{1+\alpha}} \dif y.$$
The definition \eqref{e:FraLap0} with a principle value is not convenient for use, if some suitable regularity of $f$ is further assumed, $\Delta^{\alpha/2} f(x)$ can be rewritten in a form without limit.

When $\alpha > 1$, if $f'$ and $f''$ are both bounded, then $\Delta^{\alpha/2} f(x)$ is well defined for all $x \in \R$ and can be rewritten as
\begin{equation} \label{e:FraLap}
\Delta^{\frac \alpha 2} f(x)\ =\ d_\alpha \int_\R \frac{f(x+y)-f(x)-y f'(x)}{|y|^{1+\alpha}} \dif y.
\end{equation}
Indeed, using Taylor's expansions, we easily see that
\ \ \ \
\begin{equation}
\begin{split}
\left|\int_\R \frac{f(x+y)-f(x)-y f'(x)}{|y|^{1+\alpha}} \dif y\right| & \ \le \  \left|\int_{|y| \le 1} \frac{\frac 12 f''(x+\theta_1 y) y^2}{|y|^{1+\alpha}} \dif y\right| \\
& \ \ \ \ \ \ +\left|\int_{|y| > 1} \frac{ f'(x+\theta_2 y) y-f'(x) y}{|y|^{1+\alpha}} \dif y\right| \ < \ \infty,
\end{split}
\end{equation}
where $\theta_1, \theta_2 \in (0,1)$. On the other hand,
\ \ \
\begin{equation}
\begin{split}
\int_\R \frac{f(x+y)-f(x)-y f'(x)}{|y|^{1+\alpha}} \dif y &\ = \ \lim_{r \downarrow 0} \int_{\R \setminus (-r,r)} \frac{f(x+y)-f(x)-y f'(x)}{|y|^{1+\alpha}} \dif y \\
&\ = \ \lim_{r \downarrow 0} \int_{\R \setminus (-r,r)} \frac{f(x+y)-f(x)}{|y|^{1+\alpha}} \dif y.
\end{split}
\end{equation}
In our paper, thanks to that the solution $f$ of Stein's equation has bounded first and second order derivatives, we will use the form \eqref{e:FraLap} to avoid the limit in \eqref{e:FraLap0}. Moreover, $\Delta^{\frac\alpha2} f(x)$ can be rewritten as \eqref{e:FracF-0} and \eqref{e:FracF-1} below, these two new formulations will play an important role in our analysis.
\vskip 2mm

It is well known that $\Delta^{\alpha/2}$ is the infinitesimal generator of the standard 1d symmetric $\alpha$-stable process $(Z_{t})_{t \ge 0}$ \cite{ARW00} with $Z_{0}=0$, the distribution of $Z_t$ has a density $p(t,x)$ satisfying
\ \ \
\Be  \label{e:FT}
\int_{-\infty}^\infty e^{i \lambda x} p(t,x) \dif x\ =\ e^{-t|\lambda|^\alpha},
\Ee
it is well known that $p(t,x)$ is uniquely determined by its characteristic function \cite[Section 3.3.1]{Dur10}. Note that $p(t,x)$ is called Green's function of symmetric process and satisfies the differential equation:
\ \ \
\Be  \label{e:ptEq}
\p_t p(t,x)\ =\ \Delta^{\alpha/2} p(t,x), \ \ \ \ \ p(0,x)\ =\ \delta_0(x),
\Ee
where $\delta_0(x)$ is Dirac function at $0$, i.e., $\delta_0(x)=0$ for all $x \ne 0$ and $\int_{-\infty}^\infty \delta_0(x) \dif x=1$, see \cite[(1.8)]{Kolok00b} with $A \equiv 0$ and $d=1$ therein.

Let us now consider the Orenstein-Uhlenbeck $\alpha$-stable process as the following
\ \ \
\Be \label{e:OU}
\dif X_t\ =\ -\frac{1}{\alpha   } X_t \dif t+\dif Z_t, \ \ \ \ \ X_0\ =\ x,
\Ee
we denote by $X_t(x)$ the solution to the SDE \eqref{e:OU}.
Its infinitesimal generator is
$$\mcl A f(x)\ =\ \Delta^{\frac \alpha 2} f(x)-\frac{1}{\alpha   } x f'(x) \ \ \ \ \ \ \forall \ f \in \mcl S(\R,\R),$$
where $\mcl S(\R,\R)$ is the Schwartz function space, the set of all smooth functions whose derivatives are rapidly decreasing \cite{StSh03}.
The domain $\mcl D(\mcl A)$ of the operator $\mcl A$ is the closure of $\mcl S(\R,\R)$ by a standard procedure depending on the underlying function space that we consider \cite[Chapter 2]{Part04}.

The following characterization theorem of stable distribution is well known, see \cite[Proposition 3.2]{ARW00} for instance.
\begin{thm}
Let $Y$ be a random variable. If the following equation holds:
\Be \label{de:StId}
\begin{split}
 \E \left[\Delta^{\frac \alpha2} f(Y)\right]-\frac{1}{\alpha}\E \left[Y f'(Y)\right]\ =\ 0 \ \ \ \ \ \ \forall \ \ f \in \mcl S(\R,\R),
\end{split}
\Ee
where $\alpha \in (0,2)$,
then $Y$ has a symmetric $\alpha$-stable distribution $\mu$ with
the characteristic function $e^{-  |\lambda|^\alpha}$. Moreover, the distribution of $Y$ is uniquely determined by \eqref{de:StId}.
\end{thm}
\begin{proof}
By \cite[Proposition 3.2]{ARW00} with $a_{1}=0$ and $a_{2}=\alpha d_{\alpha}$ in (3.6) therein (note that the linear operator in \cite{ARW00} is $L=\alpha \mcl A$), we get $\mu$ is the unique invariant measure of $\mcl A$ in the sense that
$$\int_{\R} \mcl A f(x) \mu(\dif x)=0.$$
See \cite[Definition 3.1]{ARW00}. This means that $Y$ has a distribution $\mu$ and this distribution is uniquely determined.
\end{proof}

For any Lipschitz function $h:\R \rightarrow \R$, Stein's equation is
\
\Be \label{e:StEq}
\Delta^{\frac \alpha 2} f(x)-\frac{1}{\alpha   } x f'(x)\ =\ h(x)-\mu(h),
\Ee
i.e.,
 \Be  \label{e:StEq-1}
 \mcl A f(x)\ =\ h(x)-\mu(h).
\Ee
It is also known that Eq. \eqref{e:StEq-1} is called Poisson equation, we can represent its solution by the stochastic process generated by $\mcl A$. More precisely,
\ \ \ \ \
 \begin{lem}  \label{l:FRep}
 	Eq. \eqref{e:StEq} has a solution
 	\Be  \label{e:FRep}
 	\begin{split}
 		f(x)& \ = \ -\int_0^\infty \int_{-\infty}^\infty p\left(1-e^{-t}, y-e^{-\frac t{  \alpha}} x\right) (h(y)-\mu(h)) \dif y \dif t,
 	\end{split}
 	\Ee
 	where $p(.,.)$ is determined by its characteristic function \eqref{e:FT}.
 \end{lem}

We will leave the proof of Lemma \ref{l:FRep} later. With the help of this lemma, we shall prove the following regularity results of $f$, which plays a crucial role in the proof of Theorem \ref{t:MainThm}.

\begin{proposition} \label{p:SolnF}
Let $f$ be the solution to Eq. \eqref{e:StEq} defined by \eqref{e:FRep}. We have the following estimates:
\
\Be  \label{e:RegF}
\|f'\| \ \le \ {\alpha } \|h'\|,
\Ee
\Be \label{e:RegF-0}
\|f''\| \ \le \  \frac 4 \pi \sqrt{\frac{2\alpha+1}{\alpha}}{\rm B}\big(\frac{\alpha-1}\alpha,\frac 2\alpha\big) \|h'\|,
\Ee
where $\|.\|$ is the uniform norm, i.e. $\|g\|=\sup_{x \in \R} |g(x)|$ for any bounded measurable function $g$.
\end{proposition}
\begin{proposition}  \label{p:SolnF-1}
For any $\gamma \in (0,1)$, we have
\Be
\sup_{x \neq y} \frac{\left|\Delta^{\frac \alpha2} f(x)-\Delta^{\frac \alpha2} f(y)\right|}{|x-y|^\gamma}  \ \le \ \frac{d_{\alpha}}{\alpha}\left[\frac{16}{\pi(2-\alpha)}\sqrt{\frac{\alpha+3}{\alpha}}+\frac{16}{\pi(\alpha-1)} \sqrt{\frac{2\alpha+1}{\alpha}}\right] {\rm B}\big(\frac{1-\gamma}{\alpha}, \frac{\gamma+\alpha}{\alpha}\big) \|h'\|.
\Ee
\end{proposition}
\ \ \ \

\subsection{Proof of Theorem \ref{t:MainThm}}
Recall that $\zeta_{n,1},...,\zeta_{n,n}$ are a sequence of independent random variables with $\E \zeta_{n,i}=0$ and $\E |\zeta_{n,i}|<\infty$ for $1 \le i \le n$. Recall the notation
$$S_n\ = \ \zeta_{n,1}+...+\zeta_{n,n}; $$
$$S_n(i)\ =\ S_n-\zeta_{n,i}, \ \ \ \ \ \ \ 1 \le i \le n.$$
\begin{lem}  \label{l:Main1}
We have
\ \ \ \
\Be
\E\left[S_n f'(S_n)\right]\ =\ \sum_{i=1}^n \int_{-N}^N \E\big[ K_i(t,N) f''(S_n(i)+t)\big] \dif t+\mcl R_1 \ \ \ \ \forall \ N>0,
\Ee
where $ K_i(t,N)\ =\ \E\left[\zeta_{n,i} 1_{\{0 \le t \le \zeta_{n,i}  \le N\}}-\zeta_{n,i} 1_{\{-N \le  \zeta_{n,i}  \le t \le 0\}}\right]$, and
\Be
\begin{split}
\mcl R_1\ =\ \sum_{i=1}^n \E\left\{\zeta_{n,i} \left[f'(S_n)-f'(S_n(i))\right] 1_{\{|\zeta_{n,i}|>N\}}\right\}  \ \ \ \ \ \forall \ N>0.
\end{split}
\Ee
\end{lem}

\begin{proof}
By the independence and $\E \zeta_{n,i}=0$ for each $i$, we have
\ \ \ \
\Be
\begin{split}
\E\left[S_n f'(S_n)\right]&\ =\ \sum_{i=1}^n \E\left[\zeta_{n,i} f'(S_n)\right]\\
&\ =\ \sum_{i=1}^n \E\left\{\zeta_{n,i} \left[f'(S_n)-f'(S_n(i))\right]\right\}\ =\ \sum_{i=1}^n I(i)+\mcl R_1,
\end{split}
\Ee
where
\Be
\begin{split}
&I(i)\ =\ \E\left\{\zeta_{n,i} \left[f'(S_n)-f'(S_n(i))\right] 1_{\{|\zeta_{n,i}| \le N\}}\right\}, \\
&\mcl R_1\ =\ \sum_{i=1}^n \E\left\{\zeta_{n,i} \left[f'(S_n)-f'(S_n(i))\right] 1_{\{|\zeta_{n,i}|>N\}}\right\}.
\end{split}
\Ee
For $I(i)$, we have
 \ \ \
\Be
\begin{split}
I(i)&\ =\ \E\left\{\zeta_{n,i} \left[f'(S_n)-f'(S_n(i))\right] 1_{\{|\zeta_{n,i}| \le N\}}\right\}  \\
&\ =\ \E\left\{\zeta_{n,i} \left[\int_0^{\zeta_{n,i}} f''(S_n(i)+t) \dif t\right] 1_{\{|\zeta_{n,i}| \le N\}}\right\} \\
&\ =\ \E\left\{\zeta_{n,i} \left[\int_{-\infty}^{\infty} f''(S_n(i)+t) \left(1_{\{0 \le t \le \zeta_{n,i}\}}-1_{\{\zeta_{n,i} \le t \le 0\}}\right) \dif t\right] 1_{\{|\zeta_{n,i}| \le N\}}\right\} \\
&\ =\ \int_{-\infty}^{\infty} \E\left[f''(S_n(i)+t) \left(1_{\{0 \le t \le \zeta_{n,i}\}}-1_{\{\zeta_{n,i} \le t \le 0\}}\right) \zeta_{n,i}  1_{\{|\zeta_{n,i}| \le N\}}\right]\dif t \\
&\ =\ \int_{-\infty}^{\infty} \E\left[f''(S_n(i)+t)\right]\E \left[\left(1_{\{0 \le t \le \zeta_{n,i}\}}-1_{\{\zeta_{n,i} \le t \le 0\}}\right) \zeta_{n,i}  1_{\{|\zeta_{n,i}| \le N\}}\right]\dif t \\
&\ =\ \int_{-N}^{N} K_i(t,N) \E\left[f''(S_n(i)+t)\right] \dif t,
\end{split}
\Ee
where the last second inequality is by the independence of $S_n(i)$ and $\zeta_{n,i}$.

Combining all the relations above, we immediately get the equality in the lemma, as desired.
\end{proof}

\begin{lem}  \label{l:Main2}
For all $x \in \R$, we have
\Be \label{e:FracF-0}
\Delta^{\frac \alpha2} f(x) \ = \ \frac{d_{\alpha}}{\alpha}\int_{-\infty}^{\infty} \frac{f'(x+z)-f'(x)}{{\rm sgn}(z)|z|^{\alpha}}\dif z.
\Ee
Moreover, for all $x \in \R$,
\Be \label{e:FracF-1}
\Delta^{\frac \alpha 2} f(x)\ =\ \int_{-N}^N \mcl K_\alpha(t,N)f''(x+t) \dif t+\mcl R_2(x),
\Ee
where $N>0$ is an arbitrary number and
\Bes
\begin{split}
& \mcl K_\alpha(t,N) \ =\ \frac{d_\alpha}{\alpha(\alpha-1)}\left(|t|^{1-\alpha}-N^{1-\alpha}\right), \\
& \mcl R_2(x)\ =\ \frac{d_\alpha}{\alpha}\int_{|z|>N} \frac{f'(x+z)-f'(x)}{{\rm sgn}(z)|z|^\alpha} \dif z.
\end{split}
\Ees
\end{lem}

\begin{proof}
We observe
\ \ \
\Bes
\begin{split}
\Delta^{\frac \alpha 2} f(x)&\ =\ d_\alpha \int_\R \frac{f(x+y)-f(x)-yf'(x)}{|y|^{1+\alpha}} \dif y \\
&\ =\ \int_\R \frac{d_\alpha}{|y|^{1+\alpha}} \int_0^y (f'(x+z)-f'(x)) \dif z\dif y \\
&\ =\ \int_{0}^\infty \frac{d_\alpha}{y^{1+\alpha}} \int_0^y (f'(x+z)-f'(x)) \dif z\dif y+\int_{-\infty}^0 \frac{d_\alpha}{(-y)^{1+\alpha}} \int_0^y (f'(x+z)-f'(x)) \dif z\dif y.
\end{split}
\Ees
It is easy to see that
\ \ \
\Bes
\begin{split}
\int_{0}^\infty \frac{d_\alpha}{y^{1+\alpha}} \int_0^y (f'(x+z)-f'(x)) \dif z\dif y&\ =\ \int_0^\infty (f'(x+z)-f'(x))\int_z^\infty \frac{d_\alpha}{y^{1+\alpha}} \dif y \dif z \\
&\ =\ \frac{d_\alpha}{\alpha}\int_0^\infty \frac{f'(x+z)-f'(x)}{z^\alpha} \dif z.
\end{split}
\Ees
Similarly,
$$\int_{-\infty}^0 \frac{d_\alpha}{(-y)^{1+\alpha}} \int_0^y (f'(x+z)-f'(x)) \dif z\dif y\ =\ -\frac{d_\alpha}{\alpha}\int_{-\infty}^0 \frac{f'(x+z)-f'(x)}{(-z)^\alpha} \dif z.$$
Combining the previous two relations, we immediately obtain \eqref{e:FracF-0}.

Now we write \eqref{e:FracF-0} as
\Bes
\begin{split}
\Delta^{\frac \alpha 2} f(x)&\ =\ \mcl J_1(x)-\mcl J_2(x)+\mcl R_2(x),
\end{split}
\Ees
with
\Be \label{e:J12R2}
\begin{split}
& \mcl J_{1}(x)\ =\ \frac{d_\alpha}{\alpha}\int_0^N \frac{f'(x+z)-f'(x)}{z^\alpha} \dif z, \\
& \mcl J_{2}(x)\ =\ \frac{d_\alpha}{\alpha}\int_{-N}^0 \frac{f'(x+z)-f'(x)}{(-z)^\alpha} \dif z, \\
& \mcl R_2(x)\ =\ \frac{d_\alpha}{\alpha}\int_N^\infty \frac{f'(x+z)-f'(x)}{z^\alpha} \dif z-\frac{d_\alpha}{\alpha}\int_{-\infty}^{-N} \frac{f'(x+z)-f'(x)}{(-z)^\alpha} \dif z.
\end{split}
\Ee
Moreover,
\Be
\begin{split}
\mcl J_{1}(x)&\ =\ \int_0^N \frac{d_\alpha}{\alpha z^\alpha}\int_0^z f''(x+t) \dif t \dif z \\
&\ =\ \int_0^\infty \int_0^z \frac{d_\alpha}{\alpha z^\alpha}f''(x+t) 1_{\{0 \le t \le z \le N\}}  \dif t \dif z \\
&\ =\ \int_0^\infty \int_t^\infty \frac{d_\alpha}{\alpha z^\alpha} 1_{\{0 \le t \le z \le N\}}  \dif z f''(x+t)\dif t \\
&\ =\ \frac{d_\alpha}{\alpha(\alpha-1)}\int_0^N \left(t^{-\alpha+1}-N^{-\alpha+1}\right)f''(x+t) \dif t.
\end{split}
\Ee
Similarly, we have
\Be
\begin{split}
\mcl J_{2}(x)\ =\ \frac{-d_\alpha}{\alpha(\alpha-1)}\int_{-N}^0 \left[(-t)^{-\alpha+1}-N^{-\alpha+1}\right]f''(x+t) \dif t.
\end{split}
\Ee
Combining the above relations of $\mcl J_1(x), \mcl J_2(x)$ and $\mcl R_2(x)$, we immediately conclude the proof.
\end{proof}

Combining Lemmas \ref{l:Main1} and \ref{l:Main2}, we prove
\begin{lem} \label{l:Main3}
The following equality holds:
\Be  \label{e:Main3}
\begin{split}
\E\left[\Delta^{\frac \alpha 2} f(S_n)-\frac{1}{\alpha  } S_n f'(S_n)\right] & \ =\ \sum_{i=1}^n \int_{-N}^N \E\left[\left(\frac{\mcl K_\alpha(t,N)} n-\frac{ K_i(t,N)}{\alpha  }\right) f''(S_n(i)+t)\right] \dif t \\
& \ \ \ \ \ \ \ -\ \frac 1\alpha \mcl R_1+\frac 1n\sum_{i=1}^n\E \big[\mcl R_2(S_n(i))\big]+\mcl R_3,
\end{split}
\Ee
where $\mcl R_1$ and $\mcl R_2(x)$ are defined in Lemmas \ref{l:Main1} and \ref{l:Main2} respectively, and
$$\mcl R_{3}\ =\ \frac 1n\sum_{i=1}^n  \E[\Delta^{\frac \alpha 2} f(S_n)-\Delta^{\frac \alpha 2} f(S_n(i))].$$
\end{lem}

\begin{proof}
Observe
\Bes
\begin{split}
\E[\Delta^{\frac \alpha 2} f(S_n)]-\frac{1}{\alpha  } \E[S_n f'(S_n)]\ =\ \frac 1n\sum_{i=1}^n  \E\left[\Delta^{\frac \alpha 2} f(S_n(i))\right]-\frac{1}{\alpha} \E\left[S_n f'(S_n)\right]+\mcl R_3.
\end{split}
\Ees
By Lemmas \ref{l:Main1} and \ref{l:Main2}, we have
\ \ \ \ \ \
\Bes
\begin{split}
& \ \ \ \ \ \ \ \frac 1n\sum_{i=1}^n  \E\left[\Delta^{\frac \alpha 2} f(S_n(i))\right]-\frac{1}{\alpha} \E\left[S_n f'(S_n)\right]\\
 &\ =\ \frac 1n\sum_{i=1}^n \E \left\{\int_{-N}^N \mcl K_{\alpha}(t,N) f''(S_{n}(i)+t) \dif t+\mcl R_2(S_n(i))\right\} \\
& \ \ \ \ \ \ \ -\frac1{\alpha  }\sum_{i=1}^n \int_{-N}^N \E\left[ K_i(t,N) f''(S_n(i)+t)\right] \dif t-\frac 1\alpha \mcl R_1 \\
& \ = \ \sum_{i=1}^{n}\int_{-N}^N \E\left[\left(\frac{\mcl K_{\alpha}(t,N)}{n}-\frac{ K_i(t,N)}{\alpha  }\right)f''(S_n(i)+t)\right]\dif t \\
& \ \ \ \ \ \ \ \ +\frac 1n\sum_{i=1}^n\E \big[\mcl R_2(S_n(i))\big]-\frac 1\alpha\mcl R_1.
\end{split}
\Ees
Hence, the lemma is proved.
\end{proof}

\begin{proof} [{\bf Proof of Theorem \ref{t:MainThm}}]
By Eq. \eqref{e:StEq}, we have
\ \
\ \ \ \
\Bes
\begin{split}
\E[h(S_n)]-\mu(h)&\ =\ \E\left[\Delta^{\frac \alpha 2} f(S_n)-\frac{1}{\alpha  } S_n f'(S_n)\right].
\end{split}
\Ees
To bound $\big|\E[h(S_n)]-\mu(h)\big|$, by Lemma \ref{l:Main3}, it suffices to bound the four terms on the right side of \eqref{e:Main3}.
By \eqref{e:RegF}, we have
\
$$ \mcl R_1 \ \le \ 2\alpha \|h'\|  \sum_{i=1}^{n}\E\left[|\zeta_{n,i}| 1_{\{|\zeta_{n,i}|>N\}}\right], $$
$$\frac 1n\sum_{i=1}^n\E \big|\mcl R_2(S_n(i))\big| \ \le \ 2d_{\alpha} \|h'\| \int_{|z|>N} \frac{1}{|z|^\alpha} \dif z \ \le \ \frac{4 d_{\alpha}}{\alpha-1} \|h'\| N^{1-\alpha}.$$
For the integral term, by \eqref{e:RegF-0} we have
\ \ \
\Bes
\begin{split}
& \ \ \ \ \ \left|\sum_{i=1}^n \int_{-N}^N \E\left[\left(\frac{\mcl K_\alpha(t,N)} n-\frac{ K_i(t,N)}{\alpha  }\right) f''(S_n(i)+t)\right] \dif t\right| \\
& \ \le \ \|f''\| \sum_{i=1}^n \int_{-N}^N \left|\frac{\mcl K_\alpha(t,N)} n-\frac{ K_i(t,N)}{\alpha  }\right| \dif t \ \le \ D_{\alpha } \|h'\| \sum_{i=1}^n \int_{-N}^N \left|\frac{\mcl K_\alpha(t,N)} n-\frac{ K_i(t,N)}{\alpha  }\right| \dif t.
\end{split}
\Ees
Finally, for $\mcl R_{3}$, by Proposition \ref{p:SolnF-1}, for all $\gamma \in (0,1)$ we have
\ \ \
\Bes
\begin{split}
|\mcl R_{3}| &\ \le \ \frac 1n\sum_{i=1}^n  \left|\E[\Delta^{\frac \alpha 2} f(S_n)-\Delta^{\frac \alpha 2} f(S_n(i))]\right| \\
&\ \le \ \frac 1n\sum_{i=1}^n  \E\left[\frac{\left|\Delta^{\frac \alpha 2} f(S_n)-\Delta^{\frac \alpha 2} f(S_n(i))\right|}{|\zeta_{n,i}|^{\gamma}} |\zeta_{n,i}|^{\gamma}\right]  \ \le \ \frac{D_{\alpha,\gamma}}n\sum_{i=1}^n   \E\left|\zeta_{n,i}\right|^{\gamma}\|h'\|.
\end{split}
\Ees
Combining the above estimates, we immediately obtain the inequality in the theorem, as desired.
\end{proof}

\section{Proofs of Theorem \ref{t:MThm2} and Corollary \ref{c:MThm2}}  \label{s:MThm2Proof}
Let us first prove Theorem \ref{t:MThm2} and then Corollary \ref{c:MThm2}, as stressed before, Corollary \ref{c:MThm2} can help us to fast determine the leading order of convergence rates, while Theorem \ref{t:MThm2} can give us an explicit bounds for $d_{W}(\mcl L(S_{n}),\mu)$.
\begin{proof}[{\bf Proof of Theorem \ref{t:MThm2}}]
It suffices to prove the inequality in the theorem by bounding the integral and the remainder $\mcl R_{N, n}$ in Theorem \ref{t:MainThm}. For the integral term, we have
\Be
\begin{split}
\sum_{i=1}^n\int_{-N}^N \left|\frac{\mcl K_\alpha(t,N)}n -\frac{ K_i(t,N)}{\alpha  }\right|
& \ = \ \frac 1{\alpha} \int_{-N}^N \left|\alpha \mcl K_{\alpha}(t,N)-n  K_{1}(t,N)\right| \dif t.
\end{split}
\Ee

Let us first estimate $\mcl R_{N, n}$, in which we need to bound the two sums. Recall $\zeta_{n,1}=\ell_n^{-\frac 1\alpha} \left(\xi_{1}-\E \xi_{1}\right)$, for the first sum, by Lemma \ref{l:MomEst},
\ \ \
\Bes
\begin{split}
\sum_{i=1}^{n} \E\big(|\zeta_{n,i}|  1_{\{|\zeta_{n,i}|>N\}}\big) 
& \ = \ n \ell^{-\frac1\alpha}_{n} \left[\ell_n^{\frac 1 \alpha}N \PP\left(|\xi_1-\E \xi_1|> \ell_n^{\frac 1 \alpha}N\right)+\int_{\ell_n^{\frac 1 \alpha}N}^\infty \PP\left(|\xi_1-\E \xi_1|>r \right) \dif  r \right] \\
& \ \le \ n  N \PP\left(|\xi_1|> \ell_n^{\frac 1 \alpha}N-|\E\xi_1|\right)+n \ell_n^{-\frac 1\alpha}\int_{\ell_n^{\frac 1 \alpha}N-|\E\xi_1|}^\infty \PP\left(|\xi_1|>r \right) \dif  r \\
&\ = \ n  N \PP\left(|\xi_1|> \ell_n^{\frac 1 \alpha}N \delta_{n}\right)+n \ell_n^{-\frac 1\alpha}\int_{\ell_n^{\frac 1 \alpha}N \delta_{n}}^\infty \PP\left(|\xi_1|>r \right) \dif  r,
\end{split}
\Ees
where $\delta_{n}=1-\ell_{n}^{-\frac 1\alpha} N^{-1} |\E \xi_{1}|$.
This and the assumption (ii') yields
\ \ \
\Bes
\begin{split}
n  N \PP\left(|\xi_1|> \ell_n^{\frac 1 \alpha}N \delta_{n}\right)& \ = \ nN \frac{\theta(1+M_{2}(\ell^{\frac 1\alpha}_{n}N \delta_{n}))}{(\ell^{\frac 1\alpha}_{n}N \delta_{n})^{\alpha}} \ = \  \frac{2d_{\alpha}(1+M_{2}(\ell_{n}^{\frac 1\alpha} N \delta_{n}))}{\alpha \delta_{n}^{\alpha}  N^{\alpha-1}}
\end{split}
\Ees
and
\Bes
\begin{split}
n \ell_n^{-\frac 1\alpha}\int_{\ell_n^{\frac 1 \alpha}N \delta_{n}}^\infty \PP\left(|\xi_1|>r \right) \dif  r &\ = \ n \ell^{-\frac1\alpha}_{n}\frac{\theta (\ell_{n}^{\frac 1\alpha} N \delta_{n})^{-\alpha+1}}{(\alpha-1)}+n \ell_n^{-\frac 1\alpha}\int_{\ell_n^{\frac 1 \alpha}N \delta_{n}}^\infty \frac{M_{2}(r)}{r^{\alpha}} \dif  r \\
&\ = \ \frac{2 d_{\alpha} (\delta_{n}N)^{1-\alpha}}{(\alpha-1)\alpha}+\frac{2 d_{\alpha}}{\alpha N^{\alpha-1}}\int_{\delta_{n}}^\infty \frac{M_{2}(r \ell_{n}^{\frac 1\alpha} N)}{r^{\alpha}} \dif  r.
\end{split}
\Ees
Therefore,
\ \
\Bes
\begin{split}
\sum_{i=1}^{n} \E\big(|\zeta_{n,i}|  1_{\{|\zeta_{n,i}|>N\}}\big)
& \ \le \ \frac{2d_{\alpha}}{\alpha \delta_{n}^{\alpha-1}} \left(\frac{1}{\alpha-1}+\frac{1}{\delta_{n}}+\frac{M_{2}(\ell_{n}^{\frac 1\alpha} N \delta_{n})}{\delta_{n}}+\int_{\delta_{n}}^\infty \frac{M_{2}(r \ell_{n}^{\frac 1\alpha} N)}{r^{\alpha}\delta^{1-\alpha}_{n}} \dif  r\right) N^{1-\alpha}.
\end{split}
\Ees
Moreover, the other sum can be bounded as follows: we immediately obtain
 \ \ \
\Bes
\frac{D_{\alpha,\gamma}}{n}\sum_{i=1}^{n}\mathbb{E}|\zeta_{n,i}|^{\gamma}=D_{\alpha,\gamma}\ell_{n}^{-\frac{\gamma}{\alpha}}\mathbb{E}|\xi_{1}-\mathbb{E}\xi_{1}|^{\gamma}.
\Ees
Combining all the estimates with the inequality in Theorem \eqref{t:MainThm}, we immediately obtain the estimate in the theorem, as desired.
\end{proof}

It is easy to verify that (i') and (ii') imply
\Be \label{e:PX1M}
\begin{split}
& \PP\left(\xi_1>x\right)\ =\ \frac{1+M_1(x)}2(1+M_2(x)) \theta x^{-\alpha},   \ \ \ \ \ \ \  \ x>A; \\
&  \PP\left(\xi_1<x\right)\ =\ \frac{1-M_1(|x|)}2(1+M_2(|x|)) \theta |x|^{-\alpha},  \ \ \ \ \ \ \ \ x<-A.
\end{split}
\Ee
\begin{proof} [{\bf Proof of Corollary \ref{c:MThm2}}]
By Theorem \ref{t:MThm2} and noticing $\ell_{n}=\frac{\alpha \theta}{2 d_{\alpha}} n$, we have
\ \ \ \
\Bes
\mcl R_{N,n} \ \le \ C_{\alpha} \left[n^{-\frac{2-\alpha}{\alpha}}+N^{1-\alpha}\right].
\Ees
It remains to bound the integral
$$\int_{-N}^{N} \big|\alpha \mcl K_{\alpha}(t,N)-n K_{1}(t,N)\big| \dif t.$$
Recall the definitions of $\mcl K_{\alpha}(t,N)$, $K_{1}(t,N)$ and $\zeta_{n,1}$, we have
\Bes
\begin{split}
& \ \ \ \ \ \ \ \int_{|t| \le 4(A+|\E \xi_{1}|)\ell_{n}^{-\frac 1\alpha}} \big|\alpha \mcl K_{\alpha}(t,N)-n K_{1}(t,N)\big| \dif t \\
& \ \le \ \frac{d_{\alpha}}{\alpha-1}\int_{|t| \le  4(A+|\E \xi_{1}|)\ell_{n}^{-\frac 1\alpha}} \frac{1}{|t|^{\alpha-1}} \dif t+n \int_{|t| \le 4(A+|\E \xi_{1}|)\ell_{n}^{-\frac 1\alpha}} \E \big|\zeta_{n,1}\big| \dif t \ \le \ C_{\alpha} n^{-\frac{2-\alpha}{\alpha}}.
\end{split}
\Ees
Now let us estimate
\Bes
\int_{|t| \ge 4(A+|\E \xi_{1}|)\ell_{n}^{-\frac 1\alpha}} \big|\alpha \mcl K_{\alpha}(t,N)-n K_{1}(t,N)\big| \dif t.
\Ees
For $t>4(A+|\E \xi_{1}|)\ell_{n}^{-\frac 1\alpha}$, we have   \ \ \
\Bes
\begin{split}
\alpha \mcl K_{\alpha}(t,N)-n K_{1}(t,N)\ =\  I_{1}-I_{2},
\end{split}
\Ees
where $b_{t}=\ell^{\frac 1\alpha}_{n} t+\E \xi_{1}$ and
\Bes
\begin{split}
& I_{1} \ = \ \frac{d_\alpha}{\alpha-1} {t^{1-\alpha}}-n \ell_n^{-\frac 1 \alpha}\left[\E\left(\xi_1 1_{\{\xi_1 \ge b_t\}}\right)-\PP\left(\xi_{1} \ge b_{t}\right) \E \xi_{1} \right], \\
& I_{2}\ =\ \frac{d_\alpha}{\alpha-1}{N^{1-\alpha}}-n \ell_n^{-\frac 1 \alpha}\left[\E\left(\xi_1 1_{\{\xi_1 \ge b_N\}}\right)-\PP\left(\xi_{1} \ge b_{N}\right) \E \xi_{1} \right].
\end{split}
\Ees
By Lemma \ref{l:MomEst} and \eqref{e:PX1M}, we have
\ \ \
\Bes
\begin{split}
\E\left[\xi_1 1_{\{\xi_1 > b_t\}}\right]
 \ = \  b_{t} \PP\left(\xi_1 > b_t\right)+\int_{b_t}^\infty \PP(\xi_1 > r) \dif r
\ =\ \frac{\alpha \theta}{2(\alpha-1)}  b_{t}^{1-\alpha}+r_{t} ,
\end{split}
\Ees
where $r_{t}$ is defined by \eqref{e:rt}.
Therefore,
\ \ \
\Bes
\begin{split}
I_{1} &\ =\ \frac{d_\alpha}{\alpha-1}t^{1-\alpha}-\frac{\alpha \theta n \ell_n^{-\frac 1\alpha} }{2(\alpha-1)} b^{1-\alpha}_{t}- n \ell_n^{-\frac 1\alpha} r_{t}+n \ell_n^{-\frac 1 \alpha}\PP\left(\xi_{1} > b_{t}\right) \E \xi_{1} \\
&\ =\ \frac{d_\alpha}{\alpha-1} t^{1-\alpha}-\frac{d_\alpha}{\alpha-1}t^{1-\alpha}\left(1+t^{-1}\ell_n^{-\frac 1\alpha} \E \xi_{1}\right)^{1-\alpha}-n \ell_n^{-\frac 1\alpha}  r_{t}+n \ell_{n}^{-\frac 1\alpha}R_t
\end{split}
\Ees
where $R_{t}$ is defined by \eqref{e:Rt}.
\vskip 2mm
As $t > 4(A+|\E \xi_{1}|)\ell_{n}^{-\frac 1\alpha}$, we have $|t^{-1}\ell_n^{-\frac 1\alpha} \E \xi_{1}| \le \frac 14$. By Taylor expansion $|(1+x)^{1-\alpha}-1| \le 4 x$ with $|x| \le \frac 14$ and the easy fact $|R_{t}| \le C t^{-\alpha} n^{-1}$, we get
\ \ \ \
$$|I_{1}| \ \le \ C_{\alpha} \left(t^{-\alpha} n^{-\frac 1\alpha}
+n^{1-\frac 1\alpha} |r_{t}|\right).$$
Similarly, we have
\ \
\Bes
|I_{2}| \ \le \ C_{\alpha} \left(N^{-\alpha} n^{-\frac 1\alpha}
+n^{1-\frac 1\alpha} |r_{N}|\right).
\Ees
Hence,
\Bes
\begin{split}
& \ \ \ \  \int_{t>4(A+|\E \xi_{1}|)\ell_{n}^{-\frac 1\alpha}} \left|\frac{d_\alpha}{\alpha-1} \left(\frac{1}{|t|^{\alpha-1}}-\frac{1}{N^{\alpha-1}}\right)-n  K_{1}(t,N)\right| \dif t \\
& \ \le \ C_{\alpha} \left(n^{-\frac{2-\alpha}{\alpha}}+N^{1-\alpha} n^{-\frac 1\alpha}+n^{1-\frac 1\alpha}\int_{t>4(A+|\E \xi_{1}|)\ell_{n}^{-\frac 1\alpha}} |r_{t}| \dif t+n^{1-\frac 1\alpha} N |r_{N}| \right).
\end{split}
\Ees
By the same argument, we get
\Bes
\begin{split}
& \ \ \ \  \int_{t<-4(A+|\E \xi_{1}|)\ell_{n}^{-\frac 1\alpha}} \left|\frac{d_\alpha}{\alpha-1} \left(\frac{1}{|t|^{\alpha-1}}-\frac{1}{N^{\alpha-1}}\right)-n  K_{1}(t,N)\right| \dif t \\
& \ \le \ C_{\alpha} \left(n^{-\frac{2-\alpha}{\alpha}}+N^{1-\alpha} n^{-\frac 1\alpha}+n^{1-\frac 1\alpha}\int_{t<-4(A+|\E \xi_{1}|)\ell_{n}^{-\frac 1\alpha}} |r_{t}| \dif t+n^{1-\frac 1\alpha} N |r_{N}| \right).
\end{split}
\Ees
Combining the previous two inequalities, we get the inequality in the corollary.
\end{proof}

\section{Proofs of Lemma \ref{l:FRep} and Propositions \ref{p:SolnF} and \ref{p:SolnF-1}}  \label{s:MPropProof}
Before proving the lemma and propositions, we first list some well known results about symmetric $\alpha$-stable process and $\Delta^{\alpha/2}$ that we shall use. 
It is easy to verify by the definition of $\Delta^{\alpha/2}$ that if $z=x-y$, then
\
\Be \label{e:Delx-y}
\Delta^{\alpha/2}_x f(x-y)\ =\ \Delta^{\alpha/2}_y f(x-y)\ =\ \Delta^{\alpha/2}_z f(z),
\Ee
where $\Delta^{\alpha/2}_x$ means that the operator $\Delta^{\alpha/2}$ acts on the variable $x$. Similarly, for $z= cx$ for some constant $c \in \R$, we have \
\Be \label{e:DelLin}
\Delta^{\alpha/2}_x f(cx)\ =\ |c|^\alpha \Delta^{\alpha/2}_z f(z).
\Ee
Recall that $p(t,x)$ is the transition probability density of standard symmetric $\alpha$-stable process $Z_{t}$, it is well known that
\ \ \ \ \
\Be \label{e:Scal}
p(t,x)=t^{-1/\alpha} p\left(1,t^{-1/\alpha} x\right), \ \ \ \ \ \ \ \ \ \ t>0, \ x \in \R.
\Ee
 We have the following estimate:
\ \ \
\begin{lem}  \label{l:HKEst}
Let $p(1,x)$ be the transition probability density of $Z_1$, we have
 \Bes
 \begin{split}
& \left|\p_x p(1,x)\right| \le \frac{1}{\alpha\pi}, \ \ \ \ \ \ \ \left|\p_x p(1,x)\right| \le \frac{2\alpha+1}{\pi} \frac1{x^{2}};\\
& \left|\p^{2}_x p(1,x)\right| \le \frac{2}{\alpha\pi}, \ \ \ \ \ \ \ \left|\p^{2}_x p(1,x)\right| \le  \frac{2\alpha+6}{\pi} \frac{1}{x^{2}}.
\end{split}
\Ees
\end{lem}
\begin{proof}
The proof is based on the inverse Fourier transform and will be given in the appendix.
\end{proof}
\begin{rem}
A sharp heat kernel estimate of $p(1,x)$ is as the following  (\cite[(1.3)]{ChZh16} and \cite[(2.11)]{ChZh16}):
\Be \label{e:HKEst-2}
\p^k_x p(1,x) \ \le  \ \frac{C_{k,\alpha} }{\left(1+|x|\right)^{\alpha+1+k}}, \ \ \ \ \ k \in \N \cup \{0\},
\Ee
but exact values of the above constants $C_{k,\alpha}$ are often difficult to be found. See \cite{ChKu03} for more details about heat kernel estimates of stable type processes.
\end{rem}
\vskip 4mm

\subsection{Proof of Lemma \ref{l:FRep}}
\begin{proof} [Proof of Lemma \ref{l:FRep}]
	Note that $\mu$ has a density $p(1,x)$, by the property $p(t,x)=t^{-1/\alpha} p(1, t^{-1/\alpha} x)$ and a change of variable, we have
	\ \ \
	\Be
	\begin{split}
		& \ \ \ \ \ \ \int_{-\infty}^\infty p\left(1-e^{-t}, y-e^{-\frac t{  \alpha}} x\right) (h(y)-\mu(h)) \dif y \\
		& \ = \ \int_{\R} p\left(1-e^{-t}, y-e^{-\frac t\alpha} x\right) h(y) \dif y-\int_{\R} p(1, y) h(y) \dif y \\
		& \ = \ \int_{\R} p(1-e^{-t}, y) h\left(y+e^{-\frac t\alpha} x\right) \dif y-\int_{\R} p(1, y) h(y) \dif y \\
		& \ = \ \int_{\R} p(1, y) h\left((1-e^{-t})^{\frac 1\alpha}y+e^{-\frac t\alpha} x\right) \dif y-\int_{\R} p(1, y) h(y) \dif y.
	\end{split}
	\Ee
	This implies
	\ \ \
	\Bes
	\begin{split}
		\left|\int_{-\infty}^\infty p\left(1-e^{-t}, y-e^{-\frac t{  \alpha}} x\right) (h(y)-\mu(h)) \dif y\right|& \ \le \ C_\alpha \|h'\| e^{-\frac t \alpha} \left(|x|+\int_{\R} |y| p(1,y ) \dif y\right) \\
		& \ \le\ C_\alpha \|h'\| e^{-\frac t \alpha} \left(|x|+1\right),
	\end{split}
	\Ees
	and hence $f(x)$ is well defined for all $x \in \R$.
	
	By Fubini theorem, we have
	\Be \label{e:FracF}
	\begin{split}
		\Delta^{\frac \alpha 2} f(x) \ = \ -\int_0^\infty \int_{-\infty}^\infty \Delta^{\frac \alpha 2}_x p\left(1-e^{-t}, y-e^{-\frac t{  \alpha}} x\right) (h(y)-\mu(h)) \dif y \dif t.
	\end{split}
	\Ee
	On the other hand, denote $s=1-e^{-t}$ and $z=y-e^{-\frac t{  \alpha}} x$, we have
	\ \
	\Bes
	\begin{split}
		\frac{\dif}{\dif t} p\left(1-e^{-t}, y-e^{-\frac t{  \alpha}} x\right)&\ = \ e^{-t}\p_s p\left(1-e^{-t}, y-e^{-\frac t{  \alpha}} x\right)+\frac 1{\alpha} e^{-\frac t\alpha} x\p_z p\left(1-e^{-t}, y-e^{-\frac t{  \alpha}} x\right) \\
		& \ = \ e^{-t}\Delta^{\frac \alpha 2}_z p\left(1-e^{-t}, y-e^{-\frac t{  \alpha}} x\right)+\frac 1{\alpha} e^{-\frac t\alpha} x\p_z p\left(1-e^{-t}, y-e^{-\frac t{  \alpha}} x\right) \\
		& \ = \ \Delta^{\frac \alpha 2}_x p\left(1-e^{-t}, y-e^{-\frac t{  \alpha}} x\right)-\frac x{\alpha} \p_x p\left(1-e^{-t}, y-e^{-\frac t{  \alpha}} x\right),
	\end{split}
	\Ees
	where the second equality is by \eqref{e:ptEq} and the third one is by
	\eqref{e:Delx-y} and \eqref{e:DelLin}. Substituting the previous relation into \eqref{e:FracF}, we get
	\Be \label{e:FracFInt}
	\begin{split}
			\Delta^{\frac \alpha 2} f(x) &\ = \ -\int_0^\infty \int_{-\infty}^\infty \frac{\dif}{\dif t} p\left(1-e^{-t}, y-e^{-\frac t{  \alpha}} x\right) (h(y)-\mu(h)) \dif y \dif t \\
			&\ \ \ \ \ \ \ \ \ -\int_0^\infty \int_{-\infty}^\infty \frac x{\alpha} \p_x p\left(1-e^{-t}, y-e^{-\frac t{  \alpha}} x\right) (h(y)-\mu(h)) \dif y \dif t \\
			&\ = \ -\int_0^\infty \int_{-\infty}^\infty \frac{\dif}{\dif t} p\left(1-e^{-t}, y-e^{-\frac t{  \alpha}} x\right) (h(y)-\mu(h)) \dif y \dif t+\frac x{\alpha} f'(x).
	\end{split}
	\Ee
By \eqref{e:ptEq}, Fubini Theorem and a straightforward calculation, we get
\ \
\Be
-\int_0^\infty \int_{-\infty}^\infty \frac{\dif}{\dif t} p\left(1-e^{-t}, y-e^{-\frac t{  \alpha}} x\right) (h(y)-\mu(h)) \dif y \dif t\ = \ h(x)-\mu(h).
\Ee
Hence, $f(x)$ solves Eq. \eqref{e:StEq}.
\end{proof}

\subsection{Proof of Proposition \ref{p:SolnF}}
In this and the next subsections, we shall often exchange differential operators and integrals without detailed proofs, since the exchangeability can be justified by a standard argument thanks to Lemma \ref{l:HKEst}.

\begin{lem} \label{l:PRep}
The density of $X_t(x)$ is $p\left( 1-e^{-t}, y-e^{-\frac t{\alpha  }} x\right)$ where $p(t,x)$ is the probability density function determined by Eq. \eqref{e:FT}.
\end{lem}
\begin{proof}
For $f \in \mcl S(\R,\R)$, define
$$Q_t f(x)\ =\ \int_{-\infty}^\infty p\left( 1-e^{-t }, y-e^{-\frac t{\alpha  }} x\right) f(y) \dif y, \ \ \ t>0.$$
We shall show that
\ \
\Be \label{e:KolEqn}
\p_t Q_t f(x)\ =\ \Delta^{\alpha/2} Q_t f(x)-\frac{1}{\alpha} x (Q_t f)'(x), \ \ \ \ Q_0 f(x)=f(x).
\Ee
Note that Eq. \eqref{e:KolEqn} is the Kolmogorov backward equation associated to SDE \eqref{e:OU}, which admits a unique solution with the form
$$Q_t f(x)\ =\ \E[f(X_t(x))].$$
Since $f \in \mcl S(\R,\R)$ is arbitrary, the probability of $X_t(x)$ has a density function as in the lemma.

It remains to prove Eq. \eqref{e:KolEqn}. $Q_0 f(x)=f(x)$ is obvious, let us now show the equation. Denote $s= 1-e^{-t}$ and $z=y-e^{-\frac t{\alpha  }} x$,  we have
\ \ \
\Be
\begin{split}
\p_t Q_t f(x)&\ =\ \p_t \int_{-\infty}^\infty p\left(s, z\right) f(y) \dif y \\
&\ =\ \int_{-\infty}^\infty \p_t p\left(s, z\right) f(y) \dif y \\
&\ =\ \int_{-\infty}^\infty e^{-t}\p_s p\left(s, z\right) f(y) \dif y\ +\ \frac{1}{\alpha  } x e^{-\frac{t}{\alpha  }}\int_{-\infty}^\infty \p_z p\left(s, z\right) f(y) \dif y.
\end{split}
\Ee
On the one hand, by \eqref{e:ptEq}, we have
\ \ \ \
\Be \label{e:PsPsz}
\begin{split}
\p_s p\left(s, z\right)&\ =\ \Delta_z^{\alpha/2} p\left(s,z\right) \\
&\ =\ d_\alpha \int_{\R} \frac{p(s,z+u)-p(s,z)}{|u|^{1+\alpha}} \dif u \\
&\ =\ d_\alpha \int_{\R} \frac{p(s,y-e^{-\frac{t}{\alpha  }} x+u)-p(s,y-e^{-\frac{t}{\alpha  }} x)}{|u|^{1+\alpha}} \dif u \\
&\ =\  e^{t} d_\alpha \int_{\R}  \frac{p(s,y-e^{-\frac{t}{\alpha  }} (x+\tl u))-p(s,y-e^{-\frac{t}{\alpha  }} x)}{|\tl u|^{1+\alpha}} \dif \tl u \\
&\ =\ e^{t} \Delta^{\alpha/2}_x p(s,y-e^{-\frac{t}{\alpha  }} x) \\
&\ =\ e^{t} \Delta^{\alpha/2}_x p(s, z),
\end{split}
\Ee
where the fourth equality is by taking $\tl u=-e^{\frac{t}{\alpha  }} u$.

On the other hand, it is easy to check
\ \ \ \
\Bes
\begin{split}
e^{-\frac{t}{\alpha  }} \p_z p\left(s, z\right)\ =\ -\p_x p\left(s, z\right).
\end{split}
\Ees
Combing the previous three relations, we immediately obtain
\ \ \
\Be
\begin{split}
\p_t Q_t f(x)&\ =\ \int_{-\infty}^\infty \Delta^{\alpha/2}_x p\left(s, z\right) f(y) \dif y-\frac{1}{\alpha  } x \int_{-\infty}^\infty \p_x p\left(s, z\right) f(y) \dif y \\
&\ =\ \Delta^{\alpha/2}_x \int_{-\infty}^\infty  p\left(s, z\right) f(y) \dif y-\frac{1}{\alpha  } x \p_x \int_{-\infty}^\infty  p\left(s, z\right) f(y) \dif y \\
&\ =\ \Delta^{\alpha/2} Q_t f(x)-\frac{1}{\alpha  } x \p_x Q_t f(x).
\end{split}
\Ee
\end{proof}

\begin{proof} [{\bf Proof of Proposition \ref{p:SolnF}}]
By Lemma \ref{l:FRep}, we have
\ \ \
\Be  \label{e:FRep-1}
f(x)\ =\ -\int_0^\infty \int_{-\infty}^\infty p\left(1-e^{-t}, y-e^{-\frac t{  \alpha}} x\right) (h(y)-\mu(h)) \dif y \dif t.
\Ee
Denote $s= 1-e^{-t}$ and $z=y-e^{-\frac{t}{\alpha  }} x$, it is easy to check
$$\p_x p(s,z)\ =\ -e^{-\frac{t}{\alpha  }} \p_z p(s,z), \ \ \ \ \ \p_y p(s,z)\ =\ \p_z p(s,z).$$
We have
\ \ \ \ \
\Be  \label{e:f'Rep}
\begin{split}
f'(x)&\ =\ \int_0^\infty \int_{-\infty}^\infty \p_x p\left(s,z\right) (h(y)-\mu(h)) \dif y \dif t \\
&\ =\ -\int_0^\infty \int_{-\infty}^\infty e^{-\frac{t}{\alpha  }} \p_y p\left(s, z\right) (h(y)-\mu(h)) \dif y \dif t\\
&\ =\ \int_0^\infty \int_{-\infty}^\infty e^{-\frac{t}{\alpha  }}  p\left(s, z\right) h'(y) \dif y \dif t.
\end{split}
\Ee
Therefore,
\Be
\begin{split}
\|f'\| & \ \le \ \|h'\| \int_0^\infty e^{-\frac{t}{\alpha  }} \int_{-\infty}^\infty   p\left(s,z\right) \dif y \dif t \\
& \ =\ \|h'\| \int_0^\infty e^{-\frac{t}{\alpha  }} \int_{-\infty}^\infty   p\left(s,z\right) \dif z \dif t \ = \ \alpha   \|h'\|.
\end{split}
\Ee
We further have
\Be
\begin{split}
f''(x)
\ =\ -\int_0^\infty \int_{-\infty}^\infty e^{-\frac{2t}{\alpha  }}  \p_z p\left(s,z\right) h'(y) \dif y \dif t.
\end{split}
\Ee
Thanks to the property $p(s,z)=s^{-1/\alpha} p(s^{-/\alpha} z)$ with $p(x)=p(1,x)$ for $x \in \R$, we have
\ \ \
\Be  \label{e:f''Rep}
\begin{split}
\|f''\| & \ \le \ \|h'\| \int_0^\infty s^{-{1\over\alpha}} e^{-\frac{2t}{\alpha  }} \int_{-\infty}^\infty  s^{-{1\over\alpha}}\left|p'\left(s^{-{1\over\alpha}}z\right)\right| \dif y \dif t.
\end{split}
\Ee
Setting $u=s^{-1/\alpha} z$ and applying the two estimates of $p'(x)$ in Lemma \ref{l:HKEst} to the two integrals $\int_{|u| \le \sqrt{\alpha(2\alpha+1)}}$ and $\int_{|u|>\sqrt{\alpha(2\alpha+1)}}$ below , we have
\ \
\Be  \label{e:Intp'Est}
\begin{split}
\int_{-\infty}^\infty  s^{-{1\over\alpha}}\left|p'\left(s^{-{1\over\alpha}}z\right)\right| \dif y
&\ = \ \int_{-\infty}^\infty  \left|p'\left(u\right)\right| \dif u \\
& \ = \ \int_{|u| \le \sqrt{\alpha(2\alpha+1)}} \frac{1}{\alpha \pi} \dif u+\int_{|u|>\sqrt{\alpha(2\alpha+1)}} \frac{2\alpha+1}{\pi u^{2}}  \dif u   \ \le\ \frac 4 \pi \sqrt{\frac{2\alpha+1}{\alpha}}.
\end{split}
\Ee
Hence,
\ \ \
\Be
\begin{split}
\|f''\| & \ \le \  \frac 4 \pi \sqrt{\frac{2\alpha+1}{\alpha}}\int_0^\infty s^{-{1\over\alpha}} e^{-\frac{2t}{\alpha  }} \dif t \|h'\|\ = \  \frac 4 \pi \sqrt{\frac{2\alpha+1}{\alpha}}{\rm B}\big(\frac{\alpha-1}\alpha,\frac 2\alpha\big) \|h'\|,
\end{split}
\Ee
where the last equality is by the change of variable $u=e^{-t}$.
\end{proof}
\subsection{Proof of Proposition \ref{p:SolnF-1}}
\begin{lem} \label{l:IBP}
Let $f \in \mcl C^2_b(\R,\R)$, the space of all second order differentiable functions with bounded zero, first, second-order derivatives. For any differentiable $h$ such that $\lim_{x \rightarrow \pm \infty} f(x) h(x)=0$, we have
\Be
\int_{-\infty}^\infty \Delta^{\frac \alpha2} f(x) h(x) \dif x\ =\ \int_{-\infty}^{\infty} \mcl I(f)(x) h'(x)\dif x,
\Ee
where
$$\mcl I(f)(x)\ =\ -\frac{d_{\alpha}}{\alpha}\int_{-\infty}^{\infty} \frac{f(x+w)-f(x)}{{\rm sgn}(w)|w|^{\alpha}}  \dif w.$$
\end{lem}
\begin{proof}
Recalling \eqref{e:FracF-0}
$$\Delta^{\frac \alpha2} f(x) \ = \ \frac{d_{\alpha}}{\alpha}\int_{-\infty}^{\infty} \frac{f'(x+z)-f'(x)}{{\rm sgn}(z)|z|^{\alpha}}\dif z,$$
and using Fubini's Theorem two times and an integration by parts, we get
\Be
\begin{split}
\int_{-\infty}^\infty \Delta^{\frac \alpha2} f(x) h(x) \dif x & \ = \ \frac{d_{\alpha}}{\alpha}\int_{-\infty}^{\infty} \int_{-\infty}^{\infty} \frac{f'(x+z)-f'(x)}{{\rm sgn}(z)|z|^{\alpha}} h(x)\dif x  \dif z \\
&\ = \ -\frac{d_{\alpha}}{\alpha}\int_{-\infty}^{\infty} \int_{-\infty}^{\infty} \frac{f(x+z)-f(x)}{{\rm sgn}(z)|z|^{\alpha}}  \dif z h'(x)\dif x.
\end{split}
\Ee
The proof is complete.
\end{proof}

\begin{proof}[{\bf Proof of Proposition \ref{p:SolnF-1}}] Observe
\Be
\begin{split}
\Delta^{\frac{\alpha}2} f(x)-\Delta^{\frac{\alpha}2} f(y) \ = \ & \int_{0}^{\infty} \int_{-\infty}^{\infty} \Delta^{\frac \alpha2}_{x} p(1-e^{-t}, z-e^{-\frac t\alpha} x)(h(z)-\mu(h)) \dif z \dif t \\
& \ - \int_{0}^{\infty} \int_{-\infty}^{\infty}\Delta^{\frac \alpha2}_{y} p(1-e^{-t}, z-e^{-\frac t\alpha} y) (h(z)-\mu(h)) \dif z \dif t.
\end{split}
\Ee
Denote $s=1-e^{-t}$ and $p(x)=p(1,x)$, we have
\ \ \
\Be
p(s, z-e^{-\frac t\alpha} x)=s^{-\frac1\alpha} p\left(s^{-\frac1\alpha}(z-e^{-\frac t\alpha} x)\right).
\Ee
By \eqref{e:Delx-y} and \eqref{e:DelLin}, we have
\ \ \
\Bes
\begin{split}
\Delta^{\frac\alpha2}_{x} p(s,z-e^{-\frac t\alpha} x)&\ = \ s^{-\frac 1\alpha} \Delta^{\frac \alpha2}_{x} p\left(s^{-\frac1\alpha} (z-e^{-\frac t\alpha} x)\right)\ = \ s^{-\frac 1\alpha} e^{-t} \Delta^{\frac \alpha2}_{z} p\left(s^{-\frac1\alpha}(z-e^{-\frac t\alpha} x)\right).
\end{split}
\Ees
Hence, by Lemma \ref{l:IBP},
\ \ \
\Be
\begin{split}
& \ \ \ \ \ \ \ \ \  \int_{0}^{\infty} \int_{-\infty}^{\infty} \Delta^{\frac\alpha2}_{x} p(s, z-e^{-\frac t\alpha} x)(h(z)-\mu(h)) \dif z \dif t  \\
&\ =\ \int_{0}^{\infty} s^{-\frac 1\alpha} e^{-t}  \int_{-\infty}^{\infty} \Delta^{\frac \alpha2}_{z} p\left(s^{-\frac1\alpha}(z-e^{-\frac t\alpha} x)\right)(h(z)-\mu(h)) \dif z \dif t  \\
&\ =\ \int_{0}^{\infty} s^{-\frac1\alpha} e^{-t}  \int_{-\infty}^{\infty} \mcl I\left(p\left(s^{-\frac 1\alpha} (\cdot-e^{-\frac t\alpha} x)\right)\right)(z)h'(z) \dif z \dif t \\
&\ = \ \int_{0}^{\infty} s^{-1+\frac1 \alpha} e^{-t}  \int_{-\infty}^{\infty} \mcl I\left(p\left(\cdot-s^{-\frac 1\alpha} e^{-\frac t\alpha}x\right)\right)(z)h'(s^{\frac 1\alpha}z) \dif z \dif t,
 \end{split}
\Ee
where the last equality is by a change of variables on $z$ and the $w$ in $\mcl I$.
Similarly,
\ \ \
\Be
\begin{split}
& \ \ \ \ \ \ \ \ \  \int_{0}^{\infty} \int_{-\infty}^{\infty} \Delta^{\frac \alpha2}_{y} p(s,z-e^{-\frac t\alpha} y)(h(z)-\mu(h)) \dif z \dif t  \\
&\ =\ \int_{0}^{\infty}  s^{-1+\frac1 \alpha} e^{-t}  \int_{-\infty}^{\infty} \mcl I\left(p\left(\cdot-s^{-\frac 1\alpha} e^{-\frac t\alpha}y\right)\right)(z)h'(s^{\frac 1\alpha}z) \dif z \dif t.
 \end{split}
\Ee
Observe
\Bes
\begin{split}
& \ \ \ \ \ \ \ \frac{1}{|x-y|^\gamma}\left|\mcl I\left(p\left(\cdot-s^{-\frac 1\alpha} e^{-\frac t\alpha}x\right)\right)(z)-\mcl I\left(p\left(\cdot-s^{-\frac 1\alpha} e^{-\frac t\alpha}y\right)\right)(z)\right| \\
&\ =\ \frac{d_{\alpha}}{\alpha} \left|\int_{-\infty}^{\infty} \frac{\delta_{w}p\left(z-s^{-\frac 1\alpha} e^{-\frac t\alpha}x\right)-\delta_{w} p\left(z-s^{-\frac 1\alpha} e^{-\frac t\alpha}y\right)}{{\rm sgn}(w)|w|^{\alpha} |x-y|^\gamma}  \dif w\right| \\
&\ = \ \frac{d_{\alpha}}{\alpha} s^{-\frac{\gamma}{\alpha}} e^{-\frac{\gamma t}{\alpha}} \left|\int_{-\infty}^{\infty} \frac{\delta_{w}p\left(z-\tl x\right)-\delta_{w} p\left(z-\tl y\right)}{{\rm sgn}(w)|w|^{\alpha} |\tl x-\tl y|^\gamma}  \dif w\right|,
 \end{split}
\Ees
where $\delta_{w} p(z)=p(z+w)-p(z)$, $\tl x=s^{-\frac 1\alpha} e^{-\frac t\alpha}x$ and $\tl y=s^{-\frac 1\alpha} e^{-\frac t\alpha}y$.
Therefore, for any $x \neq y$,
\ \ \
\Be \label{e:IntP}
\begin{split}
& \ \ \ \ \ \ \ \frac{\left|\Delta^{\frac{\alpha}2} f(x)-\Delta^{\frac{\alpha}2} f(y)\right|}{|x-y|^\gamma} \\
&\ \le \ \frac{d_{\alpha} \|h'\|}{\alpha} \int_0^\infty s^{-\frac{\gamma+\alpha-1}{\alpha}} e^{-\frac{(\gamma+\alpha) t}{\alpha}} \int_{-\infty}^\infty \left|\int_{-\infty}^{\infty} \frac{\delta_{w}p\left(z-\tl x\right)-\delta_{w} p\left(z-\tl y\right)}{{\rm sgn}(w)|w|^{\alpha} |\tl x-\tl y|^\gamma}  \dif w\right| \dif z \dif t.
 \end{split}
\Ee
Let us
bound the integral above. When $|\tl x-\tl y| \le 1$, observe
\ \ \
$$
\delta_{w}p\left(z-\tl x\right)- \delta_{w} p\left(z-\tl y\right) \ = \ \int_{0}^{w} \int_{\tl x}^{\tl y} p''\left(z+r-a\right)\dif a\dif r,
$$
$$
\delta_{w}p\left(z-\tl x\right)- \delta_{w} p\left(z-\tl y\right) \ = \ \int_{\tl x}^{\tl y} \left(p'\left(z+w-a\right)-p'\left(z-a\right)\right)\dif a,
$$
we have
\ \ \ \
\Be \label{e:IntP-1}
\begin{split}
& \ \ \ \ \ \ \int_{-\infty}^\infty \left|\int_{-\infty}^{\infty} \frac{\delta_{w}p\left(z-\tl x\right)-\delta_{w} p\left(z-\tl y\right)}{{\rm sgn}(w)|w|^{\alpha} |\tl x-\tl y|^\gamma}  \dif w\right| \dif z \\
& \ \le \  \int_{-\infty}^\infty \int_{|w| \le 1} \frac{1}{|w|^{\alpha} |\tl x-\tl y|^\gamma}  \left|\int_{0}^{w} \int_{\tl x}^{\tl y} |p''\left(z+r-a\right)|\dif a\dif r\right|  \dif w \dif z   \\
& \ \ +\int_{-\infty}^\infty \int_{|w|>1} \frac{1}{|w|^{\alpha} |\tl x-\tl y|^\gamma} \left|\int_{\tl x}^{\tl y} \left|p'\left(z+w-a\right)-p'\left(z-a\right)\right|\dif a\right|  \dif w \dif z.
 \end{split}
\Ee
Applying the two estimates of $\p^{2}_{x} p(1,x)$ in Lemma \ref{l:HKEst} to the two integrals $\int_{|z| \le \sqrt{\alpha(\alpha+3)}}$ and $\int_{|z|>\sqrt{\alpha(\alpha+3)}}$ below respectively, we obtain
\ \ \ \
\Be
\begin{split}
\int_{-\infty}^\infty  |p''\left(z+r-a\right)| \dif z & \ = \ \int_{-\infty}^\infty  |p''\left(z\right)| \dif z \\
&\ \le \ \int_{|z| \le \sqrt{\alpha(\alpha+3)}} \frac{2}{\alpha \pi} \dif z+\int_{|z|>\sqrt{\alpha(\alpha+3)}} \frac{2\alpha+6}{\pi z^{2}} \dif z \ \le\ \frac 8{\pi}\sqrt{\frac{\alpha+3}{\alpha}}.
\end{split}
\Ee
Applying the two estimates of $\p_{x} p(1,x)$ in Lemma \ref{l:HKEst} similarly, we obtain
\Bes
\begin{split}
\int_{-\infty}^\infty \left|p'\left(z+w-a\right)-p'\left(z-a\right)\right| \dif z \ & \le \ 2 \int_{-\infty}^\infty \left|p'(z)\right| \dif z \\
& \ = \ 2 \left(\int_{|z| \le \sqrt{\alpha(2\alpha+1)}}+\int_{|z|>\sqrt{\alpha(2\alpha+1)}}\right) \left|p'(z)\right| \dif z \\
& \ \le \ \frac{8}{\pi} \sqrt{\frac{2\alpha+1}{\alpha}}.
\end{split}
\Ees
Hence,  these two inequalities and \eqref{e:IntP-1}, together with Fubini's theorem, imply
\Be \label{e:IntP-2}
\begin{split}
& \ \ \ \ \ \ \int_{-\infty}^\infty \left|\int_{-\infty}^{\infty} \frac{\delta_{w}p\left(z-\tl x\right)-\delta_{w} p\left(z-\tl y\right)}{{\rm sgn}(w)|w|^{\alpha} |\tl x-\tl y|^\gamma}  \dif w\right| \dif z \\
& \ \le \  \frac 8{\pi}\sqrt{\frac{\alpha+3}{\alpha}} \int_{|w| \le 1} \frac{1}{|w|^{\alpha-1}} \dif w+\frac{8}{\pi} \sqrt{\frac{2\alpha+1}{\alpha}}\int_{|w|>1} \frac{1}{|w|^{\alpha}} \dif w \\
& \ \le \ \frac{16}{\pi(2-\alpha)}\sqrt{\frac{\alpha+3}{\alpha}}+\frac{16}{\pi(\alpha-1)} \sqrt{\frac{2\alpha+1}{\alpha}}, \ \ \ \ \ \ \ \ |\tl x-\tl y| \le 1.
 \end{split}
\Ee
When $|\tl x-\tl y|>1$, observe
\ \ \
\Be
\begin{split}
  \delta_{w}p\left(z-\tl x\right)- \delta_{w} p\left(z-\tl y\right) \ = \ \int_{0}^{w} \left(p'\left(z+r-\tl x\right)-p'\left(z+r-\tl y\right)\right)\dif r,
 \end{split}
\Ee
we have
\ \ \ \
\Bes
\begin{split}
& \ \ \ \ \ \ \int_{-\infty}^\infty \left|\int_{-\infty}^{\infty} \frac{\delta_{w}p\left(z-\tl x\right)-\delta_{w} p\left(z-\tl y\right)}{{\rm sgn}(w)|w|^{\alpha} |\tl x-\tl y|^\gamma}  \dif w\right| \dif z \\
& \ \le \  \int_{-\infty}^\infty \int_{|w| \le 1} \frac{1}{|w|^{\alpha}} \left|\int_{0}^{w} \left|p'\left(z+r-\tl x\right)-p'\left(z+r-\tl y\right)\right|\dif r\right| \dif w \dif z   \\
& \ \ +\int_{-\infty}^\infty \int_{|w|>1} \frac{1}{|w|^{\alpha}} \left[\left|\delta_{w}p\left(z-\tl x\right)\right|+\left|\delta_{w} p\left(z-\tl y\right)\right|\right] \dif w \dif z.
 \end{split}
\Ees
By a similar argument as above and $\int_\R \left|\delta_{w}p\left(z-c\right)\right| \dif z \le 2$ for any $c \in \R$, we have
\ \ \ \
\Be \label{e:IntP-3}
\begin{split}
& \ \ \ \ \ \ \int_{-\infty}^\infty \left|\int_{-\infty}^{\infty} \frac{\delta_{w}p\left(z-\tl x\right)-\delta_{w} p\left(z-\tl y\right)}{{\rm sgn}(w)|w|^{\alpha} |\tl x-\tl y|^\gamma}  \dif w\right| \dif z \\
& \ \le \ \frac{8}{\pi} \sqrt{\frac{2\alpha+1}{\alpha}}\int_{|w| \le 1} \frac{1}{|w|^{\alpha-1}} \dif w+4 \int_{|w|>1} \frac{1}{|w|^{\alpha}}  \dif w \ \le \ \frac{16}{\pi(2-\alpha)} \sqrt{\frac{2\alpha+1}{\alpha}}+\frac{8}{\alpha-1}, \ \ \ \  |\tl x-\tl y|>1.
 \end{split}
\Ee
Combining \eqref{e:IntP}, \eqref{e:IntP-2} and \eqref{e:IntP-3}, we immediately obtain
\ \ \ \
\Bes 
\begin{split}
 \frac{\left|\Delta^{\frac{\alpha}2} f(x)-\Delta^{\frac{\alpha}2} f(y)\right|}{|x-y|^\gamma} & \ \le \ \frac{d_{\alpha}}{\alpha}\left[\frac{16}{\pi(2-\alpha)}\sqrt{\frac{\alpha+3}{\alpha}}+\frac{16}{\pi(\alpha-1)} \sqrt{\frac{2\alpha+1}{\alpha}}\right] \|h'\| \int_0^\infty s^{-\frac{\gamma+\alpha-1}{\alpha}} e^{-\frac{(\gamma+\alpha) t}{\alpha}}  \dif t \\
 & \ = \ \frac{d_{\alpha}}{\alpha}\left[\frac{16}{\pi(2-\alpha)}\sqrt{\frac{\alpha+3}{\alpha}}+\frac{16}{\pi(\alpha-1)} \sqrt{\frac{2\alpha+1}{\alpha}}\right] {\rm B}\big(\frac{1-\gamma}{\alpha}, \frac{\gamma+\alpha}{\alpha}\big) \|h'\|.
 \end{split}
\Ees
\end{proof}
\section{Appendix} \label{s:App}
\subsection{Example 4: An example in \cite{JuPa98}}

Let us assume that $\xi_{1},...,\xi_{n},...$ be a sequence of i.i.d. random variables. The authors of \cite{{JuPa98}} considered the following case: $\xi_1$ has a density function as
\ \ \ \
\Be  \label{e:plog}
p(x)\ =\ K_0 \frac{(\log|x|)^\beta}{|x|^{1+\alpha}} \ \ {\rm for} \  |x|>x_0, \ \ \ \  p(x)\ =\ 0 \ \ {\rm for} \ |x| \le x_0,
\Ee
where $K_0>0$, $x_0>0$, $\alpha \in (0,2)$ and $\beta \in \R$. It is easy to check that this example is out of the scope of Theorem \ref{t:MThm2} because the conditions (i') and (ii') are not satisfied.

By \cite[Proposition 1]{JuPa98}, we have $B_{n}=0$ and $A_{n}=n^{1/\alpha} h(n)$ with $h(n)=C\log^{\frac{\gamma}{\alpha}} \left(Cn^{\frac1\alpha} \log^{\frac{\gamma}{\alpha}} (n^{\frac1 \alpha}) \right)$ and as $n \rightarrow \infty$,
$$T_n/A_n \ \Rightarrow \ \nu,$$
 where $T_{n}=\xi_{1}+...+\xi_{n}$ and $\nu$ is a symmetric stable distribution with characteristic function $\exp\left(-\frac{\alpha |\lambda|^{\alpha}}{2d_\alpha}\right)$.
The following bound was proved in \cite[Proposition 1]{JuPa98}:
$$d_{{\rm Kol}}(\mcl L(T_n/A_{n}), \nu) \ = \ O\left((\log n)^{-1}\right),$$
whose proof heavily depends on the special form of \eqref{e:plog}.
Recall $S_{n}=\left(\frac{\alpha}{2d_{\alpha}}\right)^{-\frac 1\alpha}\frac{T_{n}}{A_{n}}$,
by Remark \ref{r:Scaling}, we have
\ \ \ \
\Be \label{e:dKolLog}
d_{{\rm Kol}}(\mcl L(S_n), \mu) \ = \ O((\log n)^{-1}),
\Ee
where $\mu$ is a symmetric stable distribution with characteristic function $e^{-|\lambda|^{\alpha}}$. Applying Theorem \ref{t:MainThm}, we can prove that if \eqref{e:plog} is satisfied with $\alpha \in (1,2)$, a convergence rate $O\left((\log n)^{-1+\frac 1\alpha}\right)$ can be obtained in $W_{1}$ distance.
 \ \ \

Here we consider a new example which is more complicated than \eqref{e:plog}, more precisely,
\Be  \label{e:Log}
\PP(|\xi_{1}|>x)\ =\ \frac{K_{0}(\log x)^{\beta}}{x^{\alpha}}, \ \ \ \ \ \ \ x >x_{0}.
\Ee
Note that $K_{0}$ and $x_{0}$ here may be different from those in \eqref{e:plog}.
The corresponding density function is
$$p(x)\ =\ \frac{K_{0} \left[\alpha \left(\log |x|\right)^\beta-\beta \left(\log |x|\right)^{\beta-1}\right]}{2|x|^{\alpha+1}}, \ \ \ |x|>x_{0}; \ \ \ \ p(x)\ =\ 0, \ \ \ |x| \le x_{0}.$$
 It seems that the method in \cite{{JuPa98}} can not deal with this example directly. However, by our first main result Theorem \ref{t:MainThm}, we can prove
 \Be  \label{e:DWSn}
 d_{{\rm W}}(\mcl L(S_n), \mu) \ \le \ C_{\alpha, \beta} (\log n)^{-1+\frac 1\alpha}.
\Ee
It can be seen from the proof that \eqref{e:DWSn} also holds under
the condition \eqref{e:plog} by a similar but simpler argument. Because the proof of \eqref{e:DWSn} under the condition \eqref{e:Log} is long, we only give the leading order of the convergence.
 \vskip 3mm

Let $L, A$ be two quantities with $A>0$, if there exist some $C>0$ (which may depend on some parameters) such that
$$|L| \ \le \ C A,$$
we denote $L=O(A)$.

By Theorem \ref{t:SLConvergence}, $A_n=\inf\left\{x>0: \PP\left(|\xi_{1}|>x\right) \le \frac 1n\right\}$ can be determined by
$\frac{K_0(\log A_n)^{\beta}}{A_n^{\alpha}}\ =\ \frac 1n$,
which gives
\ \ \ \
\Be  \label{e:NoverAn}
\frac{n}{A^{\alpha}_{n}}\ = \ \frac{1}{K_0 (\log A_n)^{\beta}}.
\Ee
It is easy to see $C_{\alpha,\beta} \ n^{\frac1 \alpha} \ \le \ A_n \ \le\  C_{\alpha,\beta} \ n^{\frac1 \alpha} (\log n)^{\frac{\beta}{\alpha}}$.
By the symmetry property, $B_n = n \E \left[\xi_{1} 1_{\{|\xi_{1}| \le A_n\}}\right]=0$.

 Now we apply Theorem \ref{t:MainThm} with $N=(\log A_n)^{\frac{1}{\alpha}}$ and
 $$\zeta_{n,i}=\frac 1{\tl A_n} \xi_{i} \ \ \ \ \ \ \ \ \ \ {\rm with} \ \ \ \tl A_n=(\frac{\alpha}{2d_{\alpha}})^{1/\alpha} A_n.$$
Let us first estimate the remainder term $\mcl R_{N,n}$. Let $\gamma=2-\alpha$, we get
\ \ \ \ \
\Bes
\frac 1n \sum_{i=1}^{n}\E|\zeta_{n,i}|^{2-\alpha} \le C_\alpha \ A_n^{-2+\alpha}.
\Ees
By Lemma \ref{l:MomEst}, we get
\ \ \ \
\Bes
\begin{split}
\sum_{i=1}^{n} \E\big(|\zeta_{n,i}|  1_{\{|\zeta_{n,i}|>N\}}\big)& \  = \ \frac{n}{\tl A_n} \E\big(|\xi_{1}|  1_{\{|\xi_1|>\tl A_n N\}}\big) \\
& \ = \ nN \PP\left(|\xi_1|>\tl A_n N\right)+\frac{n}{\tl A_n} \int_{\tl A_n N}^\infty \PP(|\xi_1|>r) \dif r.
\end{split}
\Ees
By \eqref{e:NoverAn}, $N=(\log A_n)^{\frac{1}{\alpha}}$ and $A_n \ge C_{\alpha,\beta} n^{\frac1 \alpha}$,
\Bes
\begin{split}
nN \PP\left(|\xi_1|>\tl A_n N\right) \ = \ \frac{2 d_{\alpha}N^{1-\alpha}}{\alpha} \left(\frac{\log (\tl A_n N)}{\log A_n}\right)^{\beta} \ \le \ C_{\alpha,\beta} N^{1-\alpha},
\end{split}
\Ees
Moreover, by \eqref{e:NoverAn} and a change of variable $s=\frac{r}{\tl A_n N}$,
\Bes
\begin{split}
\frac{n}{\tl A_n} \int_{\tl A_n N}^\infty \PP(|\xi_1|>r) \dif r &\ =
\frac{N^{1-\alpha}}{K_0 (\log \tl A_n)^\beta} \int_{1}^\infty \frac{(\log s+\log(\tl A_nN))^{\beta}}{s^\alpha} \dif s \\
& \ = \ \frac{2 d_{\alpha}N^{1-\alpha}}{\alpha K_0} \left(\frac{\log (\tl A_n N)}{\log A_n}\right)^{\beta}\int_{1}^\infty \frac{\left(1+\frac{\log s}{\log(\tl A_n N)}\right)^{\beta}}{s^\alpha} \dif s \ \le\ C_{\alpha,\beta} N^{1-\alpha}
\end{split}
\Ees
where the inequality is by $A_n \ge C_{\alpha,\beta} n^{\frac1 \alpha}$ and an easy observation that the above integral is bounded. Hence,
\ \ \ \
\ \ \ \
\Bes
\begin{split}
\sum_{i=1}^{n} \E\big(|\zeta_{n,i}|  1_{\{|\zeta_{n,i}|>N\}}\big)\ \le \ C_{\alpha, \beta} N^{1-\alpha}.
\end{split}
\Ees
Collecting all the above estimates, we immediately obtain
\Be
\mcl R_{N,n} \ \le \ C_{\alpha, \beta} \left(A_{n}^{-2+\alpha}+\frac{1}{N^{\alpha-1}}\right) \ \le \ C_{\alpha, \beta} (\log n)^{-1+\frac 1\alpha}.
\Ee
Now let us estimate the integral term in the theorem, observe
\ \ \
\Be
\begin{split}
& \ \ \ \ \  \ \ \ \sum_{i=1}^n\int_{-N}^N \left|\frac{\mcl K_\alpha(t,N)}n -\frac{ K_i(t,N)}{\alpha   }\right| \dif t \\
&\ = \ \int_{-N}^N \left|\mcl K_{\alpha}(t,N) - \frac{n}{\alpha}  K_1(t,N)\right| \dif t\ = \ \left(\int_{|t| \le \frac{x_{0}}{\tl A_n}} +\int_{\frac{x_{0}}{\tl A_n}<|t|<N}\right) \left|\mcl K_{\alpha}(t,N) - \frac{n}{\alpha}  K_1(t,N)\right| \dif t
\end{split}
\Ee
It is easy to see that
\ \ \
\Be
\begin{split}
\int_{|t| \le \frac{x_{0}}{\tl A_n}} \left|\mcl K_{\alpha}(t,N) - \frac{n}{\alpha}  K_1(t,N)\right| \dif t \ \le \ C_{\alpha} \left(\int_{|t| \le \frac{x_0}{\tl A_n}} |t|^{1-\alpha} \dif t+\int_{|t| \le \frac{x_0}{\tl A_n}} n \tl A_n^{-1} \E|\xi_{1}|\dif t\right) \ \le \ C_{\alpha,\beta} \frac{(\log n)^{\beta}}{n^{\frac 2\alpha-1}}.
\end{split}
\Ee
We shall show below that
\ \ \ \
\Be \label{e:IntEst-1}
\begin{split}
\int_{\frac{x_0}{\tl A_n}<|t|<N} \left|\mcl K_{\alpha}(t,N) - \frac{n}{\alpha}  K_1(t,N)\right| \dif t \ \le \ C_{\alpha,\beta} \frac{N}{\log A_n}. 
\end{split}
\Ee
By $N=(\log n)^{\frac{1}{\alpha}}$ and $A_n \ge C_{\alpha,\beta} n^{\frac1 \alpha}$, we have
$$\int_{-N}^{N} \left|\mcl K_{\alpha}(t,N) - \frac{n}{\alpha}  K_1(t,N)\right| \dif t \ \le \ C_{\alpha,\beta} (\log n)^{-1+\frac 1\alpha}.$$
Combining this with that of $\mcl R_{N,n}$, we immediately obtain the estimate \eqref{e:DWSn}, as desired.

It remains to prove \eqref{e:IntEst-1}. For $t>\frac{x_0}{\tl A_n}$, we have
\ \ \ \
\Be \label{e:nKN1}
\begin{split}
n  K_1(t,N)&\ =\ n\E\left[\frac{1}{\tl A_n} \xi_{1} 1_{\{\tl A_n t \le \xi_{1} \le \tl A_nN\}}\right]\\
& \ = \ \frac{n}{\tl A_n} \left[\E\left(\xi_{1} 1_{\{\xi_{1}>\tl A_nt\}}\right)-\E\left(\xi_{1} 1_{\{\xi_{1}>\tl A_nN\}}\right)\right] \\
& \ = \ nt \PP\left(\xi_{1} > \tl A_nt\right)-nN\PP\left(\xi_{1} > \tl A_nN\right)+\frac{n}{\tl A_n} \int_{\tl A_nt}^{\tl A_nN} \PP\left(\xi_{1}>r\right) \dif r,
\end{split}
\Ee
where the last equality is by Lemma \ref{l:MomEst}. For the first term in the last line above, by \eqref{e:Log}, \eqref{e:NoverAn} and a straightforward computation, we get
\ \ \
\Bes
\begin{split}
& \  \ \ \ \ \ \ \ \ nt \PP\left(\xi_{1} > \tl A_nt\right)-nN\PP\left(\xi_{1} > \tl A_nN\right) \\
& \ = \  \frac{d_{\alpha}}{\alpha} \left[t^{1-\alpha} \left(\frac{\log (\tl A_nt)}{\log A_n}\right)^{\beta}
-N^{1-\alpha} \left(\frac{\log (\tl A_n N)}{\log A_n}\right)^{\beta}\right]  \\
& \ =\ (\alpha-1) \mcl K_{\alpha}(t,N)+O\left(\frac{\log N}{\log A_n}\right) t^{1-\alpha},
\end{split}
\Ees
where $\mcl K_{\alpha}(t,N)=\frac{d_{\alpha}}{\alpha(\alpha-1)} \left(t^{1-\alpha}-N^{1-\alpha}\right)$.
For the integral term in \eqref{e:nKN1}, as $t>\frac{x_0}{\tl A_n}$, by \eqref{e:NoverAn}, we have
\Bes
\begin{split}
\frac{n}{\tl A_n} \int_{\tl A_nt}^{\tl A_nN} \PP\left(\xi_{1}>r\right) \dif r& \ = \ \frac{nK_0}{2\tl A_n} \int_{\tl A_nt}^{\tl A_nN} \frac{(\log r)^{\beta}}{r^{\alpha}} \dif r \\
&\ = \ \frac{n K_0}{2\tl A_n} \int_{\tl A_nt}^{\tl A_nN} \frac{(\log \tl A_n)^{\beta}}{r^{\alpha}} \dif r + \frac{n K_0}{2\tl A_n} \int_{\tl A_nt}^{\tl A_nN} \frac{(\log r)^{\beta}-(\log \tl A_n)^{\beta}}{r^{\alpha}} \dif r \\
&\ = \ \mcl K_{\alpha}(t,N)+O\left(\frac{\log N}{\log A_n}\right) t^{1-\alpha}  + \frac{n K_0}{2\tl A_n} \int_{\tl A_nt}^{\tl A_nN} \frac{(\log r)^{\beta}-(\log \tl A_n)^{\beta}}{r^{\alpha}} \dif r.
\end{split}
\Ees
From the previous estimate, it is easy to check
\Be
\mcl K_{\alpha}(t,N)-\frac{n K_{N,1}(t,N)}{\alpha}\ =\ \frac{K_0 n}{2 \alpha \tl A_n} \int_{\tl A_nt}^{\tl A_nN} \frac{(\log r)^{\beta}-(\log \tl A_n)^{\beta}}{r^{\alpha}} \dif r+O\left(\frac{\log N}{\log A_n}\right) t^{1-\alpha}.
\Ee
When $t>\frac{x_0}{A_n}$, we first observe
\ \ \ \
\Be
\begin{split}
\left|(\log r)^{\beta}-(\log \tl A_n)^{\beta}\right| & \ \le \ \left|\left(\log \tl A_n+\log \frac{r}{\tl A_n}\right)^\beta-(\log \tl A_n)^\beta\right| \\
& \ \le \ (\log \tl A_n)^\beta \left|\left(1+\frac{\log \frac{r}{\tl A_n}}{\log \tl A_n}\right)^\beta-1\right| \\
& \ \le \ C_{\alpha,\beta}   (\log A_n)^{\beta-1} \left|\log \frac{r}{\tl A_n}\right|,
\end{split}
\Ee
where the last inequality is by Taylor's expansion and the easy fact $\left|\frac{\log \frac{r}{\tl A_n}}{\log \tl A_n}\right|<1$ when $1 \le r \le N \tl A_n$.
The previous two relations, \eqref{e:NoverAn} and a change of variable $s=r/\tl A_n$ yield
\Be
\begin{split}
\left|\mcl K_\alpha(t,N)-\frac{n}{\alpha}  K_1(t,N)\right| &\ \le\  C_{\alpha,\beta} \left[\frac{n (\log A_n)^{\beta-1} K_0}{2 A^{\alpha}_n} \int_{t}^{N} \frac{\left|\log s\right|}{s^{\alpha}} \dif s+\frac{\log N}{\log A_n} t^{1-\alpha}\right] \\
&\ =\  C_{\alpha,\beta} \left[(\log A_n)^{-1} \int_{t}^{N} \frac{\left|\log s\right|}{s^{\alpha}} \dif s+\frac{\log N}{\log A_n} t^{1-\alpha}\right] \\
\end{split}
\Ee
It is easy to check when $t \ge 1$,
\ \ \
\Be
\begin{split}
\int_{t}^{N} \frac{\left|\log s\right|}{s^{\alpha}} \dif s \le \frac{C_\alpha}{(\alpha-1)^2},
\end{split}
\Ee
when $0< t \le 1$, we have
\ \ \
\Be
\begin{split}
\int_{t}^{N} \frac{\left|\log s\right|}{s^{\alpha}} \dif s\  = \ \int_{t}^{1} \frac{\left|\log s\right|}{s^{\alpha}} \dif s+\int_{1}^{N} \frac{\left|\log s\right|}{s^{\alpha}} \dif s \ \le \ \frac{ |\log t| t^{1-\alpha}}{\alpha-1}+\frac{2t^{1-\alpha}}{(\alpha-1)^2}.
\end{split}
\Ee
Collecting the above estimates, we get
\Be
\left|\mcl K_\alpha(t,N)-\frac{n}{\alpha}  K_1(t,N)\right| \ \le  \ C_{\alpha, \beta}\begin{cases}  (\log A_n)^{-1}\left(1+t^{1-\alpha}\log N\right), \ \ \ &1 \le t \le N; \\
(\log A_n)^{-1} t^{1-\alpha} \left(1+|\log t|+\log N\right), \ \ \ &\frac{x_0}{\tl A_n} \le t<1.
 \end{cases}
\Ee
Hence,
$$\int_{x_{0}/\tl A_n}^N \left|\mcl K_\alpha(t,N)-\frac{n}{\alpha}  K_1(t,N)\right| \dif t \ \le \ C_{\alpha,\beta} \left(\frac{N}{\log A_n}+\frac{N^{2-\alpha} \log N}{\log A_n}+\frac{1}{\log A_{n}}\right) \ \le\ C_{\alpha,\beta} \frac{N}{\log A_n}.$$
By the same argument, we get
$$\int_{-N}^{-x_{0}/\tl A_n} \left|\mcl K_\alpha(t,N)-\frac{n}{\alpha}  K_1(t,N)\right| \dif t \ \le \ C_{\alpha,\beta} \frac{N}{\log A_n}. $$
Hence, \eqref{e:IntEst-1} is proved.

\subsection{Proof of Lemma \ref{l:HKEst}}
For notational simplicity, we write $p(x)=p(1,x)$. Due to the symmetry property $p(x)=p(-x)$ for all $x \in \R$, it suffices to consider $p(x)$ for $x \ge 0$. We shall frequently use the easy relations
$$\Gamma(z+1)\ = z \Gamma(z) \ \ \ \forall \ z \in \R; \ \ \ \ \ \Gamma(z) \le 1 \ \ \ \forall \ z \in (1,2).$$
For $\theta \in (-1,\infty)$, we denote
$$I_{\theta}(x)\ = \ \int_{0}^{\infty} \lambda^{\theta} e^{-\lambda^{\alpha}} \cos(\lambda x) \dif \lambda, \ \ \ \ \ \ \ \ J_{\theta}(x)\ = \ \int_{0}^{\infty} \lambda^{\theta} e^{-\lambda^{\alpha}} \sin(\lambda x) \dif \lambda.$$
It is easy to verify by the easy estimate $|\cos(\lambda x)|, |\sin(\lambda x)| \le 1$ and a change of variable $t=\lambda^{\alpha}$ that
$$\left|I_{\theta}(x)\right| \ \le \ \frac{\Gamma(\frac{\theta+1}{\alpha})}{\alpha}, \ \ \ \left|J_{\theta}(x)\right| \ \le \ \frac{\Gamma(\frac{\theta+1}{\alpha})}{\alpha}.$$
By the inverse of Fourier transform, we have
\Be
\begin{split}
p(x)&\ = \ \frac{1}{2 \pi} \int_{-\infty}^{\infty} e^{-|\lambda|^\alpha} e^{-i x\lambda} \dif \lambda  \ = \ \frac{1}{2 \pi} \int_{-\infty}^{\infty} e^{-|\lambda|^\alpha} \cos (\lambda x)  \dif \lambda   \ = \ \frac{I_{0}(x)}{\pi}.
\end{split}
\Ee
Hence,
\Be
\begin{split}
p(x)  \ \le \ \frac{\Gamma(\frac1\alpha)}{\pi \alpha}\ = \ \frac{\Gamma(1+\frac 1\alpha)}{\pi} \ \le \ \frac{1}{\pi}.
\end{split}
\Ee
For $x>0$, using integration by parts two times, we get
\Be  \label{e:ExaHK-1}
\begin{split}
I_{0}(x)  \ = \ \frac{\alpha(\alpha-1)I_{\alpha-2}(x)-\alpha^{2}I_{2\alpha-2}(x)}{x^{2}}.
\end{split}
\Ee
Moreover,
\ \
\Be
\begin{split}
\left|I_{\alpha-2}(x)\right|  & \ \le \ \frac{\Gamma(1-\frac 1\alpha)}{\alpha} \ = \ \frac{\Gamma(2-\frac 1\alpha)}{\alpha-1} \ \le \frac{1}{\alpha-1},
\end{split}
\Ee
\Be  
\begin{split}
\left|I_{2\alpha-2}(x)\right| & \ \le \ \frac{\Gamma(2-\frac 1\alpha)}{\alpha} \ \le \ \frac{1}{\alpha}.
\end{split}
\Ee
Hence,
\Be
p(x) \ = \ \frac{I_{0}(x)}{\pi} \ \le \ \frac{2 \alpha}{\pi x^{2}}.
\Ee
Now we estimate $p'(x)$. It is obvious that
\ \ \
\Bes
I_{0}'(x) \ = \ -J_{1}(x)
\Ees
and
thus
\Bes
|I_{0}'(x)|  \ \le \ \frac{\Gamma(\frac 2\alpha)}{\alpha}\ \le \ \frac{1}{\alpha}.
\Ees
Hence,
$$|p'(x)| \ = \ \frac{|I_{0}'(x)|}{\pi} \ \le \frac{1}{\pi \alpha}.$$
For $x>0$, using integration by parts, we have
\Be \label{e:EstHK-0}
I_{0}(x) \ = \ \frac{\alpha J_{\alpha-1}(x)}{x},
\Ee
which implies
\Be
I_{0}'(x) \ = \ -\frac{\alpha J_{\alpha-1}(x)}{x^{2}}+\frac{\alpha I_{\alpha}(x)}{x}.
\Ee
It is easy to check
\ \ \
\Be
\left|\frac{\alpha J_{\alpha-1}(x)}{x^{2}}\right| \ \le \ \frac{1}{x^{2}}.
\Ee
Using integration by parts we have
\ \ \
\Be
I_{\alpha}(x)\ = \ \frac{\alpha J_{2\alpha-1}(x)-\alpha J_{\alpha-1}(x)}{x},
\Ee
which gives
\ \ \
\Be
\left|\frac{\alpha I_{\alpha}(x)}{x}\right|\ \le \ \frac{\alpha^{2}}{x^{2}}  \left(\frac{\Gamma(2)}{\alpha}+\frac{\Gamma(1)}{\alpha}\right)\ = \ \frac{2\alpha}{x^{2}}.
\Ee
Hence,
$$|p'(x)| \ \le \ \frac{|I_{0}'(x)|}{\pi} \ \le \ \frac{(2\alpha+1)}{\pi x^{2}}.$$
For $p''(x)$, we have
\Be
p''(x) \ = \ -\frac{I_{2}(x)}{\pi},
\Ee
which immediately implies
$$|p''(x)| \ \le \ \frac{\Gamma(\frac 3 \alpha)}{\alpha\pi} \ \le \ \frac{2}{\alpha \pi}.$$
Using integration by parts two times,
\Be
I_{2}(x)\ = \ -\frac 2{x^{2}} I_{0}(x)+\frac{\alpha^{2}+3\alpha}{x^{2}} I_{\alpha}(x)-\frac{\alpha^{2}}{x^{2}} I_{2\alpha}(x).
\Ee
By a similar computation, we get the second estimate of $p''(x)$, as desired.
\bibliographystyle{amsplain}

\end{document}